\documentclass[11pt, oneside]{amsart}

\usepackage{amsmath,amsfonts,amssymb,amsthm,mathrsfs}
\usepackage{xcolor}
\usepackage{hyperref}

\hypersetup{
  colorlinks=true,
  linkcolor=red,
  citecolor=blue,
  urlcolor=black
}

\newenvironment{acknowledgements}
  {\par\medskip\noindent\textbf{Acknowledgements.}\itshape}
  {\par\medskip}

\theoremstyle{plain}
\newtheorem{thm}{Theorem}[section]
\newtheorem{lem}[thm]{Lemma}

\newtheorem{prop}[thm]{Proposition}

\theoremstyle{definition}
\newtheorem{definition}[thm]{Definition}

\newtheorem{rem}[thm]{Remark}

\newcommand{\ind}{\mathbf 1}

\newcommand{\D}{\textnormal{D}}
\newcommand{\R}{\mathbb{R}}

\renewcommand*\D{\mathop{}\!\mathrm{D}}

\newcommand{\<}{\langle}
\renewcommand{\>}{\rangle}

\newcommand{\diam}{\textnormal{diam}}

\newcommand{\clball}{\bar{\mathcal{B}}}

\renewcommand{\H}{\mathcal{H}}

\newcommand{\hidethis}[1]{}

\newcommand{\vol}{\textnormal{vol}}

\newcommand{\ric}{\operatorname{Ric}}

\newcommand{\met}{\mathrm g}
\newcommand{\man}{M}
\newcommand{\Tr}{\operatorname{Tr}}

\newcommand{\Id}{\operatorname{Id}}

\newcommand{\PR}{\mathcal P^{\textnormal{ac}}_c(\man)}

\newcommand{\Sec}{\operatorname{sec}}

\newcommand{\ET}{ \mathbf H}

\newcommand{\curv}{ \mathrm{curv}}

\newcommand{\Rs}{\mathbf R}

\newcommand{\hess}{\operatorname{Hess}}
\newcommand{\der}{\operatorname{D}}

\newcommand{\dd}{n}

\begin{document}
\title[]{Alexandrov spaces with non negative curvature and displacement convexity of the entropy tensor}

\author{Jordan Serres}
\thanks{Sorbonne Université, LPSM, 4 place Jussieu 75005 Paris, France\\ ({\sf jordan.serres@sorbonne-universite.fr})}

\begin{abstract}
On a smooth Riemannian manifold, Aishwarya, Rotem and Shenfeld characterised
nonnegative sectional curvature as the matrix displacement convexity of an entropy
tensor, the Lagrangian, matrix-valued refinement of Shenfeld's entropy matrix. In order to extend the entropy tensor to a finite-dimensional Alexandrov space of curvature bounded below, we construct a parallel trivialisation satisfying both the cocycle property and the second variation formula. The construction is strongly inspired by Petrunin's synthetic parallel transport. The entropy tensor defined is taken in block-diagonal form; 
on smooth manifolds the resulting convexity property still characterises nonnegative sectional curvature exactly.
We show that the smooth equivalence persists
synthetically: an Alexandrov space has nonnegative curvature if and only if its entropy tensor is matrix displacement convex. 
\end{abstract}

\maketitle

\section*{Introduction}

Aishwarya, Rotem and Shenfeld \cite{aishwarya2025sectional} proved that a smooth complete
Riemannian manifold $(M,\met)$ has nonnegative \emph{sectional} curvature if and only if its
\emph{entropy tensor} is matrix displacement convex along $\mathscr W_2$-geodesics. The
entropy tensor $\ET_t^{\mu_0\to\mu_1}(x)=-\int_0^t\hess_{\gamma(s)}\theta_s\,ds$ is the
Lagrangian, matrix-valued refinement of the entropy matrix of Shenfeld \cite{Shenfeld24};
its trace, integrated against the source measure, returns the Boltzmann entropy, so that
matrix displacement convexity upgrades the classical Lott--Villani--Sturm \cite{LottVillani09, Sturm06(i)} equivalence ``$\ric\ge0\iff$ scalar
entropy displacement convex'' from the trace to the full tensor, and thereby from Ricci to
sectional curvature. 

Since the Alexandrov condition $\curv\ge0$ is a synthetic lower \emph{sectional} bound, it is natural to ask whether the matrix equivalence of
\cite{aishwarya2025sectional} survives the loss of smoothness. 
We construct first, inspired by Petrunin's parallel transportation construction \cite{petrunin1998parallel}, a family of isometries between tangent spaces at regular points (see Proposition \ref{prop:cocycle}), satisfying both a \emph{cocycle property} and the second variation formula \eqref{SVperp}. 
Thanks to the cocycle property, we can define, on a finite-dimensional Alexandrov space, the entropy and the block-diagonal entropy tensors (see Definition~\ref{def:alexentropytensor}) and prove:

\medskip
\noindent\textbf{Theorem 1} (Theorem~\ref{thm:aleximpliesdisplacementconvex}, informally).
\emph{If an Alexandrov space has nonnegative curvature, then its block-diagonal entropy tensor is matrix displacement convex.}

\medskip
\noindent\textbf{Theorem 2} (Theorem~\ref{thm:displacementconveximpliesAlex}, informally).
\emph{Conversely, if the block-diagonal entropy tensor of an Alexandrov space (a priori of curvature bounded below by some $\kappa\in\R$) is matrix displacement convex, then the space has nonnegative curvature (\textit{i.e.}\ $\kappa$ can be taken to
be zero).}

\medskip
The forward direction follows Petrunin's argument \cite{petrunin2010alexandrov} that showed that a space with curvature bounded below in the Alexandrov sense satisfies the Lott--Villani--Sturm condition, but we never take traces, keeping all inequalities at the level of quadratic forms on the tangent
spaces. The key point in this proof is the second variation formula \eqref{SVperp}.

The converse does not use the extraction of the sectional curvature from Jacobi field expansion as in the smooth setting \cite{aishwarya2025sectional}; instead it focuses a transport
onto a shrinking ball, runs a backward matrix Riccati comparison, and produces the
squared-distance comparison that characterises $\curv\ge0$ via Hu--Su--Wang \cite{HuSuWang}. In particular, translated back to smooth Riemannian geometry, the Alexandrov strategy would therefore constitute an independent alternative proof of the converse implication in Aishwarya--Rotem--Shenfeld.

\begin{acknowledgements}
The author thanks Gautam Aishwarya, Rotem Assouline, Alexandros Eskenazis, Liran Rotem, and Yair Shenfeld for insightful discussions and help with the project.
The author is also grateful to Anton Petrunin for pointing out the reference \cite{petrunin1998parallel}.

The LLM Claude was used as an aid for exploring technical difficulties and conceptual ideas throughout the project. The human author remains solely responsible for the content.
\end{acknowledgements}

\section{Preliminaries and definition of the entropy tensor}\label{sec:prelim}

\subsection{Preliminaries on Alexandrov geometry}\label{ss:prelimalex}

We collect here, without proofs, the notions and results from Alexandrov geometry
that are used in the sequel. Standard references for the general theory are
\cite{BuragoBuragoIvanov01,BuragoGromovPerelman92,AlexanderKapovitchPetrunin}.
Throughout, $(X,d)$ denotes a complete, locally compact, finite-dimensional
geodesic space.

\medskip
\noindent\textbf{Spaces with curvature bounded below.}
For $\kappa\in\R$ let $M_\kappa^2$ be the model surface of constant curvature $\kappa$.
Given three points $p,q,r\in X$ with $d(p,q)+d(q,r)+d(r,p)<2\pi/\sqrt{\kappa}$
(no restriction if $\kappa\le0$), a \emph{comparison triangle}
$\triangle\tilde p\tilde q\tilde r\subset M_\kappa^2$ is a triangle with the same side
lengths. The space $(X,d)$ has \emph{curvature bounded below by $\kappa$}
($\curv\ge\kappa$, or $\mathrm{CBB}(\kappa)$) if it is geodesic and every point has a
neighbourhood $U$ such that, for every geodesic triangle $\triangle pqr\subset U$ and
every point $s$ on the side $[qr]$, one has
\[
d(p,s)\ \ge\ d(\tilde p,\tilde s),
\]
where $\tilde s\in[\tilde q\tilde r]$ is the point with $d(\tilde q,\tilde s)=d(q,s)$.
Equivalently, comparison angles $\widetilde\angle pqr$ are monotone, or the Toponogov
hinge comparison holds. By the Globalization (Toponogov) Theorem, a local lower bound is
automatically global \cite{BuragoGromovPerelman92,AlexanderKapovitchPetrunin}. Typical
examples are complete Riemannian manifolds with $\Sec\ge\kappa$, their Gromov--Hausdorff
limits, convex hypersurfaces, quotients $M/G$ of nonnegatively curved manifolds by
isometric group actions, and spherical/Euclidean cones over $\mathrm{CBB}(1)$ spaces.
We write $\dd=\dim X$ for the Hausdorff dimension, which for a $\mathrm{CBB}$ space is an
integer, and $\mu_H=\H^{\dd}$ for the $\dd$-dimensional Hausdorff measure.

\medskip
\noindent\textbf{Tangent cones and spaces of directions.}
For $p\in X$, the \emph{space of directions} $\Sigma_pX$ is the completion of the set of
geodesic directions at $p$ equipped with the angle metric; it is a compact
$\mathrm{CBB}(1)$ space of dimension $\dd-1$. The \emph{tangent cone} is the Euclidean
cone $T_pX:=C(\Sigma_pX)$, and it is canonically isometric to the pointed
Gromov--Hausdorff blow-up $\lim_{\lambda\to\infty}(\lambda X,p)$. The tangent cone is a
$\mathrm{CBB}(0)$ space, and the exponential map $\exp_p:T_pX\to X$ is defined by following
geodesics; it is the natural domain for the first- and second-order calculus below.

\medskip
\noindent\textbf{Regular and singular points.}
A point $p\in X$ is \emph{regular} if $T_pX$ is isometric to the Euclidean space
$\R^{\dd}$, equivalently if $\Sigma_pX$ is isometric to the unit sphere
$\mathbb S^{\dd-1}$; otherwise $p$ is \emph{singular}. Let
$S_X=\{p\in X:T_pX\not\cong\R^{\dd}\}$ be the singular set. Then $S_X$ has
$\mu_H$-measure zero; indeed $S_X$ has Hausdorff dimension $<\dd-1$
\cite[Thm.~A]{OtsuShioya94}, and in fact, for interior (non-boundary) singular points,
Hausdorff dimension at most $\dd-2$, i.e.\ codimension at least two
\cite{BuragoGromovPerelman92}. The regular set
$X\setminus S_X$ is thus of full $\mu_H$-measure and carries, by Otsu--Shioya, a
$C^0$-Riemannian metric $\met$ inducing $d$ \cite{OtsuShioya94,otsu30differential}. This is the structure underlying
all the matrix-valued objects below.

\medskip
\noindent\textbf{Semiconcave and DC functions.}
A locally Lipschitz function $f:X\to\R$ is \emph{$\lambda$-concave} if its restriction to
every unit-speed geodesic $\gamma$ satisfies $(f\circ\gamma)''\le\lambda$ in the
distributional (support) sense, and \emph{semiconcave} if it is $\lambda$-concave locally
for some $\lambda$. The squared distance $\tfrac12 d(\cdot,q)^2$ is locally semiconcave,
and so is any Kantorovich potential for the cost $\tfrac12 d^2$. The natural regularity
class for second-order calculus is that of \emph{DC functions} (differences of
semiconcave functions): on the regular part the DC functions form an atlas, with respect
to which $\met$ has locally $BV$ coefficients, and every semiconcave function is twice
differentiable $\mu_H$-almost everywhere \cite{Perelman_DC,otsu30differential}.

\medskip
\noindent\textbf{The Hessian of a semiconcave function.}
Let $f$ be semiconcave. By Alexandrov's theorem in the DC charts
\cite{Perelman_DC,otsu30differential}, there is a set $\operatorname{Reg}f\subset
X\setminus S_X$ of full $\mu_H$-measure such that, for every $p\in\operatorname{Reg}f$,
$f$ is differentiable at $p$ and admits a second-order Taylor expansion: there is a
symmetric bilinear form $\hess_pf$ on $T_pX\cong\R^{\dd}$ with
\begin{equation}\label{eq:hessdefprelim}
f(\exp_p w)=f(p)+d_pf(w)+\tfrac12\,\hess_pf(w,w)+o(|w|^2),\qquad w\in T_pX.
\end{equation}
The form $\hess_pf$ is defined only $\mu_H$-a.e., is symmetric, and is bounded above
(resp.\ below) wherever $f$ is $\lambda$-concave (resp.\ semiconvex). We stress that
$\hess_pf$ is a bilinear form on the \emph{individual} tangent space $T_pX$; comparing or
integrating such forms at different points requires the parallel transport recalled below.

\medskip
\noindent\textbf{Gradient curves and the Hamilton--Jacobi flow.}
For a semiconcave $f$ and $p\in X$ the gradient $\nabla_pf\in T_pX$ is defined as the
element dual to the differential of $f$ in the sense of \cite{petrunin2010alexandrov,
PerelmanPetruninQG}; the \emph{gradient curve} $\dot\alpha=\nabla f(\alpha)$ exists,
is unique and is $1$-Lipschitz in the appropriate sense. For a $\tfrac12 d^2$-Kantorovich
potential $\theta$, its \emph{Hamilton--Jacobi shifts}
\begin{equation}\label{def:theta}
\theta_t(x):=\inf_{y\in X}\Big\{\theta(y)+\frac{1}{2t}\,d(x,y)^2\Big\}
\end{equation}
form, for $t>0$, a family of $\tfrac1t$-concave functions solving the Hamilton--Jacobi
equation $\partial_t\theta_t+\tfrac12|\nabla\theta_t|^2=0$, and satisfying the semigroup
identity $\theta_{t_1}=H_{t_1-t_0}\theta_{t_0}$ for $0<t_0<t_1$
\cite[\S2]{petrunin2010alexandrov}. The optimal transport rays
$\gamma(s)=\exp_x(-s\nabla\theta(x))$ are minimizing geodesics and coincide with the
characteristics of \eqref{def:theta}; along them $\dot\gamma(s)=-\nabla\theta_s(\gamma(s))$
\cite[\S1]{petrunin2010alexandrov}.

\medskip
\noindent\textbf{Parallel transport and the second variation formula.}
In a $\mathrm{CBB}$ space, Petrunin \cite{petrunin1998parallel} constructed, for each fixed pair of interior
points $p,q$ of $\gamma$, a \emph{synthetic} parallel transport: a linear isometry
$T_pX\to T_qX$ arising as a limit of compositions of comparison configurations, which
satisfies a synthetic second variation formula \cite[Thm.~1.1.B]{petrunin1998parallel};
in the nonnegatively curved case the latter bounds the squared distance between
exponential variations issued at $p$ and $q$ from above by its Euclidean model value (we
use it in the ultrafilter form of \cite[\S1]{petrunin2010alexandrov}, see Section~\ref{sec:forward}). The construction is pairwise: it depends on auxiliary limits, uniqueness fails ``in any good sense''
\cite[Introduction]{petrunin1998parallel}, hence no compatibility (cocycle property)
between the transports attached to different pairs is provided by \cite{petrunin1998parallel}.

In order to be well-defined, the matrix-valued objects of this paper require a transport satisfying the cocycle property. Such a trivialisation is provided by Lemma \ref{lem:existadmissible} through a coboundary frame that defines the entropy tensor, but note that its convexity in such a frame is not a geometric statement. 
Moreover, for the proof of Theorem \ref{thm:aleximpliesdisplacementconvex}, we will also need the second variation formula.

To get a transport enjoying both properties, we \emph{construct}, in
Section~\ref{ss:selection}, a selection of Petrunin-type transports with an exact
cocycle on a dense set of times \emph{and} a second-variation bound in the directions
normal to the ray. The second-variation bound in mixed (tangential--normal) directions
remains open (Remark~\ref{rem:fullSV}); accordingly, the entropy tensor of
Definition~\ref{def:alexentropytensor} is taken in block-diagonal form with respect to
the canonical tangential--normal splitting along the ray, for which the normal bound,
together with an exact tangential computation, suffices for
Theorem~\ref{thm:aleximpliesdisplacementconvex}.

Note that the $C^0$-Riemannian connection of \cite{OtsuShioya94} has, in the DC charts of \cite{Perelman_DC}, Christoffel symbols of only Radon-measure regularity \cite{ambrosio2018dc}, so its parallel transport cannot serve as the geometric trivialisation underlying the second-variation formula. This is why it does not suffice for the forward direction (Theorem \ref{thm:aleximpliesdisplacementconvex}).

\subsection{Definition of the entropy tensor}\label{ss:entropytensor}\,\\

\noindent\textbf{The Riemannian prototype.}
Our object of study is the Alexandrov analogue of the \emph{entropy tensor} introduced by
Aishwarya, Rotem and Shenfeld \cite{aishwarya2025sectional} on Riemannian manifolds,
itself a Lagrangian, pointwise refinement of the entropy matrix of Shenfeld
\cite{Shenfeld24}. We recall the construction and, above all, the precise sense in which
it refines the scalar entropy, since that is what dictates the correct definition in the
singular setting.

Let $(M,\met)$ be a smooth complete Riemannian manifold, $\mathrm{vol}$ its volume measure,
and
\[
H(\mu):=\int_M\log\frac{d\mu}{d\mathrm{vol}}\,d\mu
\]
the Boltzmann entropy of $\mu\ll\mathrm{vol}$. Given $\mu_0,\mu_1\in\mathcal
P_c^{\mathrm{ac}}(M)$, the Brenier--McCann theorem yields a $\tfrac12 d^2$-concave
potential $\theta$ such that $t\mapsto\mu_t:=(\exp(t\nabla\theta))_\#\mu_0$ is the
$\mathscr W_2$-geodesic from $\mu_0$ to $\mu_1$. Along the transport geodesic issued from
$x$ one forms the Jacobi field matrix $\mathbf J_s(x)$ (in a parallel orthonormal frame,
with $\mathbf J_0=\Id$, $\dot{\mathbf J}_0=\hess_x\theta$), the matrix
$\mathbf U_s(x):=\dot{\mathbf J}_s(x)\mathbf J_s(x)^{-1}=\hess_{\exp_x(s\nabla\theta)}\theta_s$,
and the \emph{entropy tensor}
\begin{equation}\label{eq:ETprototype}
\ET_t^{\mu_0\to\mu_1}(x):=-\int_0^t\mathbf U_s(x)\,ds,
\end{equation}
a symmetric $\dd\times\dd$ matrix attached to each $x$
\cite[Def.~2.2]{aishwarya2025sectional}. The entropy tensor is said to be \emph{matrix
displacement convex} if, for all $\mu_0,\mu_1\in\mathcal
P_c^{\mathrm{ac}}(M)$,
\[
\ddot\ET_t^{\mu_0\to\mu_1}(x)\ \succeq\ \big(\dot\ET_t^{\mu_0\to\mu_1}(x)\big)^2
\]
in the Loewner order. The result that
motivates this entire paper is the equivalence
\begin{equation}\label{eq:ARS}
\Sec(M)\ge0\quad\Longleftrightarrow\quad\ET\ \text{is matrix displacement convex}
\end{equation}
\cite[Thm.~2.10]{aishwarya2025sectional}: nonnegative \emph{sectional} curvature is captured exactly by the matrix (untraced) convexity of $\ET$. This
should be contrasted with the classical theorem that $\ric\ge0$ is equivalent to
displacement convexity of the \emph{scalar} entropy $H$
\cite{Cordero-ErasquinMcCannSchmuckenschlager01,vonRenesseSturm05}.
\\

\noindent \textbf{The matrix point of view of Shenfeld.}
The matrix refinement originates in \cite{Shenfeld24}, where, for a density flow
$(\rho_t,\theta_t)$ solving a continuity/Hamilton--Jacobi system, the scalar entropy
production $S(t)=\partial_t E(t)$ of $E(t)=\int\rho_t\log\rho_t$ is lifted to the
\emph{entropy production matrix}
$\mathcal S(t)=\int\nabla\rho_t\otimes_S\nabla\theta_t\,dx$ and the \emph{entropy matrix}
$\mathcal E(t)=\int_0^t\mathcal S(s)\,ds$, so that $E(t)=E(0)+\Tr[\mathcal E(t)]$ and the
matrix differential inequality $\partial_t\mathcal S\succeq\mathcal S^2$ encodes
intrinsic-dimensional information invisible to its trace. The crucial structural point,
emphasised in \cite{Shenfeld24}, is that the Hessian of the Kantorovich potential
$\nabla^2\theta_t$ is precisely the matrix that, upon tracing, returns the scalar entropy
production via the Jacobi formula, while \emph{before} tracing it resolves the entropy
production direction by direction in space.\\

\noindent\textbf{Entropy tensor versus scalar entropy.}
The relationship between the entropy tensor and the usual scalar entropy is the point on
which the whole construction turns, and it is more delicate than ``take the trace''. The
exact statement is the change-of-variables identity \cite[Lem.~2.5]{aishwarya2025sectional}
\begin{equation}\label{eq:traceentropy}
H(\mu_t)\ =\ H(\mu_0)+\int_M\Tr\big[\ET_t^{\mu_0\to\mu_1}(x)\big]\,d\mu_0(x),
\end{equation}
which follows from the Jacobi formula
$\tfrac{d}{ds}\log\det\mathbf J_s=\Tr[\mathbf U_s]=-\tfrac{d}{ds}\Tr[\ET_s]$. Thus the
scalar entropy is recovered from the entropy tensor by \emph{two} operations: a trace over
directions in $T_xM$, \emph{and} an integration over $x$ against the source measure
$\mu_0$. The pointwise trace $\Tr[\ET_t(x)]=-\log\det\mathbf J_t(x)$
is merely the local logarithmic Jacobian of the transport map; only its $\mu_0$-average is
an entropy. Consequently the entropy tensor carries strictly more information than the
scalar entropy in two independent ways: it resolves the directional, off-trace structure of
$\hess\theta_t$ (this is what \eqref{eq:ARS} exploits, and what distinguishes sectional from
Ricci curvature), and it is local in $x$ rather than spatially averaged.

\emph{The Euclidean case.} On $\R^{\dd}$ all tangent spaces are canonically identified, the
transport map is $x\mapsto x+t\nabla\theta(x)$, and $\ET_t(x)$ is an honest symmetric matrix
field. Even here, recovering $H$ from $\ET$ is the trace \emph{and} the spatial average
\eqref{eq:traceentropy}: the distinction between the pointwise matrix and the globally
integrated entropy matrix $\mathcal E(t)$ of \cite{Shenfeld24}, for which $E(t)=E(0)+\Tr[\mathcal
E(t)]$ with the spatial integral already absorbed, is visible already in flat space.

\emph{The Riemannian case.} On $(M,\met)$ the matrix $\mathbf U_s(x)=\hess_{\gamma(s)}\theta_s$
is a bilinear form on the \emph{varying} tangent space $T_{\gamma(s)}M$. Forming the integral
\eqref{eq:ETprototype} as a single matrix therefore requires trivialising $TM$ along $\gamma$
by Levi--Civita parallel transport. The trace $\Tr[\mathbf U_s]$ is frame-independent, so the
scalar identity \eqref{eq:traceentropy} is unaffected by the choice of trivialisation; but the
matrix $\ET_t(x)$ itself, and a fortiori its Loewner order, depend on the connection. This is
why the convexity \eqref{eq:ARS} sees the full curvature operator (sectional curvature) and not
only its trace (Ricci): the parallel transport that is invisible to the trace is exactly the
datum that the matrix statement constrains.

\emph{The Alexandrov case.} Three further degradations occur. First, $\theta_t$ is only semiconcave, so
$\mathbf U_s=\hess_{\gamma(s)}\theta_s$ exists only for $\mu_H$-a.e.\ ray $\gamma$ and a.e.\
parameter $s$, as a bilinear form on $T_{\gamma(s)}X$. Second, that tangent space is Euclidean, so that ordering bilinear forms is meaningful, only at regular points, which however form a
set of full measure. Third, the parallel transport along $\gamma$ that is needed to assemble
\eqref{eq:ETprototype} into a single matrix, and to differentiate it covariantly, must
simultaneously carry a cocycle property and a second-variation bound; no transport in
the literature is known to carry both (Section~\ref{ss:prelimalex}), and a dedicated
construction --- a selection of Petrunin-type transports with an exact cocycle on a
dense set of times and a second-variation bound in the directions normal to the ray ---
is carried out in Section~\ref{ss:selection}. Since the mixed-direction bound remains
open for that construction (Remark~\ref{rem:fullSV}), the definition below records
only the tangential scalar and the normal block of $\hess\theta_s$, in block-diagonal
form with respect to the canonical splitting
$T_{\gamma(s)}X=\R\dot\gamma(s)\oplus\dot\gamma(s)^\perp$; the discarded mixed entries
never enter the trace. As in the Riemannian case the trace is intrinsic: the scalar
identity \eqref{eq:traceentropy} holds verbatim and is exactly the computation by which
Petrunin \cite[Prop.~2.2]{petrunin2010alexandrov} establishes displacement convexity of the
scalar entropy, using only $\Tr[\mathbf U_s]$ and requiring no parallel transport. \\

\noindent\textbf{The definition.}
We can now give the definition of the (block-diagonal) entropy tensor in Alexandrov spaces.

\begin{definition}\label{def:alexentropytensor}
Let $(X,d)$ be a finite-dimensional Alexandrov space with $\curv\ge\kappa$ for some
$\kappa>-\infty$, of dimension $\dd$, and let $\mu_H=\H^{\dd}$ be its Hausdorff measure. Let
$\mu_0,\mu_1\in\PR$ be compactly supported probability measures absolutely continuous with
respect to $\mu_H$, and let $\theta:X\to\R$ be the $\tfrac12 d^2$-Kantorovich potential of
\cite[Thm.~1.1]{bertrand2008existence} transporting $\mu_0$ to $\mu_1$. Denote by
$(\theta_t)_{t>0}$ its Hamilton--Jacobi shifts \eqref{def:theta}, and, for
$\mu_H$-a.e.\ starting point $x$, by $\gamma(s)=\exp_x(-s\nabla\theta(x))$ the transport ray.
For $\mu_H$-a.e.\ such $x$ and a.e.\ $s>0$ the point $\gamma(s)$ is regular and lies in
$\operatorname{Reg}\theta_s$, and we set
\[
U_s(x):=\hess_{\gamma(s)}\theta_s,
\]
a symmetric bilinear form on $T_{\gamma(s)}X\cong\R^{\dd}$, defined by the expansion
\eqref{eq:hessdefprelim}.

At every regular point $\gamma(s)$ of the ray the tangent space splits orthogonally as
\[
T_{\gamma(s)}X\ =\ \R\dot\gamma(s)\ \oplus\ L_{\gamma(s)},
\qquad L_{\gamma(s)}:=\dot\gamma(s)^{\perp},
\]
a splitting that is canonical (it depends only on the ray). Decompose $U_s(x)$
accordingly into the \emph{tangential scalar} and the \emph{normal block},
\[
c_s(x):=\frac{U_s(x)\big(\dot\gamma(s),\dot\gamma(s)\big)}{|\dot\gamma(s)|^2}\in\R,
\qquad
V_s(x):=U_s(x)\big|_{L_{\gamma(s)}\times L_{\gamma(s)}},
\]
discarding the mixed entries $U_s(x)(\dot\gamma(s),\cdot)|_{L_{\gamma(s)}}$.

Fix along $\mu_H$-a.e.\ ray $\gamma$ an \emph{admissible
parallel trivialisation}: a family of linear isometries
$\tau_{s\to t}:T_{\gamma(s)}X\to T_{\gamma(t)}X$, defined for $s,t$ in a dense set
$D_\gamma\subset(0,\infty)$ of parameters, satisfying the exact cocycle identity
$\tau_{r\to t}\circ\tau_{s\to r}=\tau_{s\to t}$ ($s<r<t$ in $D_\gamma$), preserving the
splitting (i.e.\ $\tau_{s\to t}\,\dot\gamma(s)=\dot\gamma(t)$, hence
$\tau_{s\to t}L_{\gamma(s)}=L_{\gamma(t)}$), and such that
the matrices of the normal blocks $V_s(x)$, $s\in D_\gamma$, in a common $\tau$-parallel
orthonormal frame of $L_{\gamma(\cdot)}$ admit a locally bounded measurable extension to
a.e.\ $s>0$, still denoted $V_s(x)$ (for the trivialisation constructed in
Section~\ref{ss:selection} this extension exists and is canonical, by monotonicity:
Lemma~\ref{lem:monextension}). The
\emph{entropy tensor} associated with these data is then the block-diagonal matrix
\[
\ET_t^{\mu_0\to\mu_1}(x)\ :=\ -\int_0^t
\begin{pmatrix} c_s(x) & 0\\[2pt] 0 & V_s(x)\end{pmatrix}\,ds,
\]
a symmetric $\dd\times\dd$ matrix defined for $\mu_H$-a.e.\ $x$ and every $t>0$. We say that
$\ET$ is \emph{matrix displacement convex} if, for all such $\mu_0,\mu_1$ and for $\mu_H$-a.e.\ $x$,
\begin{equation}\label{eq:MDCdef}
\ddot\ET_t^{\mu_0\to\mu_1}(x)\ \succeq\ \big(\dot\ET_t^{\mu_0\to\mu_1}(x)\big)^2
\end{equation}
in the Loewner order, where $\dot\ET_t=\tfrac{\D}{dt}\ET_t$ denotes the covariant derivative
along $\gamma$, understood distributionally in $t$. Since $\dot\ET_t$ is block-diagonal,
so is its square, and \eqref{eq:MDCdef} decouples into the scalar inequality
$\dot c_t\le -c_t^2$ and the normal-block inequality $\dot V_t\preceq -V_t^2$, both
understood distributionally in $t$.
\end{definition}

\begin{rem}\label{rem:perdirection}
The per-direction statement that $t\mapsto e^{-\langle w(t),\ET_t(x)w(t)\rangle}$ is concave for
every parallel field $w$ is, by \cite[Lem.~2.7]{aishwarya2025sectional}, equivalent to the scalar
inequality $\tfrac{\D^2}{dt^2}\langle w,\ET_t w\rangle\ge\big(\tfrac{\D}{dt}\langle
w,\ET_t w\rangle\big)^2$, i.e.\ to $\langle w,\ddot\ET_t w\rangle\ge\langle w,\dot\ET_t
w\rangle^2$. It is implied by \eqref{eq:MDCdef} but strictly weaker than it, since
$\langle w,(\dot\ET_t)^2 w\rangle=|\dot\ET_t w|^2\ge\langle w,\dot\ET_t w\rangle^2$ by
Cauchy--Schwarz; the gap is the Jensen gap discussed in \cite[Rem.~2.8]{aishwarya2025sectional}.
\end{rem}

\begin{rem}\label{rem:traceisscalar}
Taking the trace in Definition~\ref{def:alexentropytensor} and integrating against $\mu_0$
recovers the scalar entropy through \eqref{eq:traceentropy}; this is the only operation that is
insensitive to the choice of parallel transport, and it reduces the matrix construction to
Petrunin's scalar computation \cite[Prop.~2.2]{petrunin2010alexandrov}. Note that the
block-diagonal reduction is invisible to the trace:
$\Tr\big[\operatorname{diag}(c_s,V_s)\big]=c_s+\Tr V_s=\Tr U_s$, since the discarded
mixed entries are off-diagonal in any splitting-adapted frame. The content of the
matrix statement lies entirely in the off-trace, frame-dependent part of $\ET_t$.
\end{rem}

\begin{rem}\label{rem:framechoice}
As in \cite[Rem.~2.4]{aishwarya2025sectional}, the matrix $\ET_t^{\mu_0\to\mu_1}(x)$ depends on
the choice of measurable orthonormal frame trivialising $T_{\gamma(\cdot)}X$, but the convexity
property of Definition~\ref{def:alexentropytensor} does not: any two such frames differ by a
constant orthogonal matrix along $\gamma$ (parallelism is preserved), under which both
$\ddot\ET_t$ and $(\dot\ET_t)^2$ transform by conjugation, so the Loewner inequality is
frame-independent \emph{for a fixed trivialisation $\tau$}. It does, however, depend a
priori on the trivialisation itself: two cocycle trivialisations differ by a gauge
$t\mapsto O_t\in O(\dd)$ which need not be constant in $t$, and \eqref{eq:MDCdef} is not
invariant under such gauges. The trace (Remark~\ref{rem:traceisscalar}) and, more
generally, the ordered spectrum of $\dot\ET_t$ are gauge-invariant; the untraced
statement of Theorem~\ref{thm:aleximpliesdisplacementconvex} is asserted for the explicit
trivialisation constructed in Section~\ref{ss:selection}.
\end{rem}

\subsection{Block-diagonal versus full entropy tensor}\label{rem:blockvsfull}\,\\

In this section, we compare the block-diagonal entropy tensor of Definition \ref{def:alexentropytensor} to the full entropy tensor defined by the quantity
\[
-\int_0^t U_s(x)\,ds,
\]
which exactly correspond to Shenfeld's entropy matrix on smooth Riemannian manifolds.
\medskip\\
\noindent\textbf{On smooth manifolds.}
On a smooth Riemannian manifold both tensors are defined, and it is worth recording
precisely how Definition~\ref{def:alexentropytensor} relates to the original tensor
\eqref{eq:ETprototype} of \cite{aishwarya2025sectional}. Work along a fixed transport
ray $\gamma$, in a parallel orthonormal frame whose first vector is
$\dot\gamma/|\dot\gamma|$, and write the full Hessian in block form
\[
\mathbf U_s\ =\ \begin{pmatrix} c_s & b_s^{\!\top}\\ b_s & V_s\end{pmatrix},
\qquad b_s:=\hess_{\gamma(s)}\theta_s\big(\tfrac{\dot\gamma}{|\dot\gamma|},\cdot\big)\Big|_{L_{\gamma(s)}}.
\]

\emph{(1) As tensors, the two objects differ.} The mixed entries $b_s$ are generically
nonzero: $b_s=0$ for all $s$ exactly when $\nabla\theta_s$ is an eigenvector of
$\hess\theta_s$ along the ray. This does hold for radial (focusing) transports, whose
potentials are functions of the distance to a point, by the eikonal equation; it fails
for a generic potential. The traces, however, always agree
(Remark~\ref{rem:traceisscalar}), so the scalar identity \eqref{eq:traceentropy} is the
same for both tensors.

\emph{(2) For a fixed transport, the block-diagonal inequality is formally weaker.}
The smooth Riccati equation $\dot{\mathbf U}_s+\mathbf U_s^2+R_s=0$, with
$R_s=R(\cdot,\dot\gamma)\dot\gamma$ (so that $R_s$ is itself block-diagonal, with
vanishing tangential entry and normal block $R_\perp$), reads componentwise
\[
\dot c_s=-\big(c_s^2+|b_s|^2\big)\ \le\ -c_s^2,
\qquad
\dot V_s+V_s^2\ =\ -R_\perp-\,b_s\otimes b_s .
\]
Thus the tangential scalar inequality of \eqref{eq:MDCdef} holds \emph{automatically},
with slack $|b_s|^2$, while the normal-block inequality is equivalent to
$R_\perp+b_s\otimes b_s\succeq0$; the full inequality
$\dot{\mathbf U}\preceq-\mathbf U^2$ is equivalent to $R_\perp\succeq0$. For a single
transport, block-diagonal convexity is therefore implied by full convexity (concretely:
the pinching $A\mapsto\operatorname{diag}(c,V)$ is order-preserving and satisfies
$\operatorname{pinch}(\mathbf U^2)\succeq(\operatorname{pinch}\mathbf U)^2$), and is in
general strictly weaker, the deficit being $b_s\otimes b_s\succeq0$.

\emph{(3) Quantified over all transports, the two convexity properties coincide.} Each
is equivalent to $\Sec\ge0$: full matrix displacement convexity by
\cite[Thm.~2.10]{aishwarya2025sectional}, and block-diagonal matrix displacement
convexity because, on one hand, it is implied by the full one (point (2)), and, on the
other hand, it still forces $\Sec\ge0$, by a direct radial argument available in the
smooth setting: given a $2$-plane $\operatorname{span}(u,w)$ at $z$, take the unit-speed
geodesic $\gamma$ with $\dot\gamma(0)=u$, a point $p=\gamma(\ell)$ \emph{before} the
cut point of $z$ along $\gamma$, and a transport whose potential is (a multiple of)
$\tfrac12d^2(\cdot,p)$ near $z$; the potential and its Hamilton--Jacobi shifts are
smooth functions of the distance to $p$ there, so $b_s\equiv0$ \emph{exactly} along
the ray through $z$ (eikonal equation), the deficit in point (2) vanishes, and the
normal-block inequality returns $\<R_\perp w,w\>=\Sec(u,w)|w|^2\ge0$. In particular,
on smooth manifolds, block-diagonal and full matrix displacement convexity (each
quantified over all transports) are both equivalent to $\Sec\ge0$, the block-diagonal
characterisation having the formally weaker convexity hypothesis in the converse
direction. We caution that this direct radial argument is genuinely smooth: it uses
potentials supported near a single ray and the smoothness of $d^2(\cdot,p)$ before
the cut locus; for the Alexandrov analogue and its current status see
Remark~\ref{rem:blockconvimpliesfullconv}.

\medskip
\noindent\textbf{On Alexandrov spaces}
Remark~\ref{rem:blockvsfull} asserted, on a smooth manifold and for a single transport,
that block-diagonal matrix displacement convexity is implied by the full one, the
mechanism being that the block-diagonal compression is order-preserving and
superquadratic. We record this rigorously, in a form that also covers the synthetic
setting, since it makes precise the sense in which the block-diagonal tensor of
Definition~\ref{def:alexentropytensor} is a genuine relaxation of the full entropy tensor.

Fix the orthogonal splitting $\R^{\dd}=\R e_1\oplus E$, $E:=e_1^{\perp}$ (in the
applications $e_1=\dot\gamma/|\dot\gamma|$ and $E=L_{\gamma(\cdot)}$), and let
\[
P:\mathrm{Sym}(\dd)\to\mathrm{Sym}(\dd),\qquad
P\begin{pmatrix}A_{11}&A_{12}\\ A_{12}^{\top}&A_{22}\end{pmatrix}
=\begin{pmatrix}A_{11}&0\\ 0&A_{22}\end{pmatrix},
\]
be the pinching onto block-diagonal form, so that $P(U)=\operatorname{diag}(c,V)$ with
$c=U(e_1,e_1)$ and $V=U|_{E\times E}$ in the notation of
Definition~\ref{def:alexentropytensor}.

\begin{lem}\label{lem:pinching}
For every $A\in\mathrm{Sym}(\dd)$ one has $A\succeq0\Rightarrow P(A)\succeq0$, and
\[
P(A^{2})-P(A)^{2}=\begin{pmatrix}A_{12}A_{12}^{\top}&0\\ 0&A_{12}^{\top}A_{12}\end{pmatrix}\ \succeq\ 0 .
\]
In the splitting $\R e_1\oplus E$, writing $A_{12}=b^{\top}$ with $b\in E$, the deficit
is $P(A^{2})-P(A)^{2}=\operatorname{diag}\bigl(|b|^{2},\,b\otimes b\bigr)$.
\end{lem}

\begin{proof}
With $J:=\operatorname{diag}(1,-\Id_{E})\in O(\dd)$ one has $P(A)=\tfrac12(A+JAJ)$; if
$A\succeq0$ then $JAJ\succeq0$, hence $P(A)\succeq0$ (so $P$ is a unital positive linear
map). For the second identity, the diagonal blocks of $A^{2}$ are
$A_{11}^{2}+A_{12}A_{12}^{\top}$ and $A_{22}^{2}+A_{12}^{\top}A_{12}$, while those of
$P(A)^{2}$ are $A_{11}^{2}$ and $A_{22}^{2}$; subtracting gives the displayed matrix,
which is $\succeq0$ since $A_{12}A_{12}^{\top}$ and $A_{12}^{\top}A_{12}$ are Gram
matrices. In the splitting $A_{11}=c$ is a scalar and $A_{12}=b^{\top}$ a row, so
$A_{12}A_{12}^{\top}=|b|^{2}$ and $A_{12}^{\top}A_{12}=b\otimes b$.
\end{proof}

We can now show that the matrix displacement convexity property of the full entropy tensor implies the same property for the block-diagonal entropy tensor.

\begin{prop}\label{prop:fulltoblock}
Let $\gamma$ be a transport ray carrying an admissible parallel trivialisation
$\widehat W$ in the sense of Definition~\ref{def:alexentropytensor}, read in a
$\widehat W$-parallel splitting-adapted orthonormal frame, and suppose the full Hessian
matrix $U_t=\hess_{\gamma(t)}\theta_t$ is locally bounded (e.g.\ by the interior bound
\eqref{eq:twosidedhess}, which is fulfilled in CBB($\kappa$), $\kappa\in\R$) and of bounded variation in $t$ in this frame. If the full
entropy tensor $-\int_0^tU_s\,ds$ is matrix displacement convex along $\gamma$, i.e.
\[
\dot U_t\ \preceq\ -\,U_t^{2}\qquad\text{distributionally in }t,
\]
then the block-diagonal entropy tensor
$\ET_t=-\int_0^t\operatorname{diag}(c_s,V_s)\,ds$ is matrix displacement convex along $\gamma$:
$\dot{\hat U}_t\preceq-\hat U_t^{2}$ distributionally, with $\hat U_t=P(U_t)$.
\end{prop}

\begin{proof}
The frame being splitting-adapted, its first vector is $\dot\gamma/|\dot\gamma|$ at every
time, so the splitting $\R e_1\oplus E$ is constant in the frame and $P$ is the
\emph{same} constant projection of $\mathrm{Sym}(\dd)$ at all $t$. Consequently $P$
commutes with $\tfrac{d}{dt}$, $\hat U_t=P(U_t)$ is locally bounded $BV$, and
$P(\dot U_t)=\dot{\hat U}_t$ as distributions.

Fix $\eta\in\R^{\dd}$. Since $P$ is self-adjoint for the trace inner product,
$\langle\hat U_t\eta,\eta\rangle=\Tr\bigl(U_t\,P(\eta\eta^{\top})\bigr)$, and
$P(\eta\eta^{\top})\succeq0$ by Lemma~\ref{lem:pinching}; write its spectral
decomposition $P(\eta\eta^{\top})=\sum_k\mu_k\,\zeta_k\zeta_k^{\top}$ with $\mu_k\ge0$.
Then, as distributions in $t$,
\begin{align*}
\tfrac{d}{dt}\langle\hat U_t\eta,\eta\rangle
&=\Tr\bigl(\dot U_t\,P(\eta\eta^{\top})\bigr)\\
&=\sum_k\mu_k\,\tfrac{d}{dt}\langle U_t\zeta_k,\zeta_k\rangle\\
& \le\ -\sum_k\mu_k\,|U_t\zeta_k|^{2}\\
&=-\,\langle P(U_t^{2})\eta,\eta\rangle ,
\end{align*}
the inequality being the full-convexity hypothesis $\dot U_t\preceq-U_t^2$ tested on each
$\zeta_k$ and combined with the nonnegative weights $\mu_k$. At a.e.\ $t$ the point
$\gamma(t)$ is regular, so $U_t$ is a genuine symmetric matrix and
Lemma~\ref{lem:pinching} gives $P(U_t^{2})\succeq P(U_t)^{2}=\hat U_t^{2}$ pointwise; the
exceptional times are Lebesgue-null. Both $\langle P(U_t^2)\eta,\eta\rangle$ and
$\langle\hat U_t^{2}\eta,\eta\rangle$ being locally integrable ($U_t$ bounded), this a.e.\
inequality is also distributional, and chaining gives
$\tfrac{d}{dt}\langle\hat U_t\eta,\eta\rangle\le-\langle\hat U_t^{2}\eta,\eta\rangle$
distributionally, for every $\eta$, i.e.\ $\dot{\hat U}_t\preceq-\hat U_t^{2}$.
\end{proof}

\begin{rem}[The deficit and its curvature meaning]\label{rem:pinchdeficit}
On a smooth manifold $U$ solves the Riccati equation $\dot U_s=-U_s^{2}-R_s$ with
$R_s=R(\,\cdot\,,\dot\gamma)\dot\gamma=\operatorname{diag}(0,R_\perp)$ block-diagonal
(since $R(\dot\gamma,\dot\gamma)\dot\gamma=0$). Applying $P$ and
Lemma~\ref{lem:pinching} gives the exact identity
\[
\dot{\hat U}_s=-\hat U_s^{2}-\operatorname{diag}\bigl(|b_s|^{2},\,R_\perp+b_s\otimes b_s\bigr),
\qquad
b_s:=\hess_{\gamma(s)}\theta_s\bigl(\tfrac{\dot\gamma}{|\dot\gamma|},\cdot\bigr)\Big|_{E},
\]
so that full convexity $\Leftrightarrow R_\perp\succeq0$ while block-diagonal convexity
$\Leftrightarrow R_\perp+b_s\otimes b_s\succeq0$; since $b_s\otimes b_s\succeq0$ the
former implies the latter, the slack being $b_s\otimes b_s$ in the normal block and
$|b_s|^{2}$ in the tangential scalar. This reproves the smooth case and identifies the deficit. 
\end{rem}

\begin{rem}\label{rem:blockconvimpliesfullconv}
The converse implication ``block-diagonal convexity \emph{recovers} the full one''
holds on smooth manifolds, but only when both are quantified over all transports,
and the proof is forced through the fact that they are both equivalent to the nonnegative sectional curvature. In the Alexandrov case, this converse implication is equivalent to Property~(SV), which we leave as an open problem, see Remark~\ref{rem:fullSV}.
Note that this converse implication fails for a single transport, regardless of the regularity of the space.
\end{rem}

\section{Alexandrov spaces are matrix displacement convex}\label{sec:forward}

The goal of this section is to prove the forward direction of \eqref{eq:ARS} in the Alexandrov synthetic setting.

\begin{thm}\label{thm:aleximpliesdisplacementconvex}
Let $(X,d)$ be an Alexandrov space with non-negative curvature, and let $\mu_0$ and
$\mu_1$ be compactly supported probability measures absolutely continuous with respect
to the Hausdorff measure. Then, along
$\mu_H$-a.e.\ transport ray, the trivialisation constructed in
Section~\ref{ss:selection} is admissible in the sense of
Definition~\ref{def:alexentropytensor}, and the associated block-diagonal entropy tensor
$\ET_t^{\mu_0\to\mu_1}$ is matrix displacement convex.
\end{thm}

\subsection{An admissible Petrunin-type parallel trivialisation}\label{ss:selection}

The proof of Theorem \ref{thm:aleximpliesdisplacementconvex} requires, along $\mu_H$-a.e.\ transport ray, a single family of isometries
between the tangent spaces enjoying simultaneously (a)~an exact cocycle property, which
makes the matrix-valued quantities of Definition~\ref{def:alexentropytensor} well defined
across all times, and (b)~Petrunin's second-variation bound, which drives the Hessian
comparison of Step~3 of Section \ref{sec:proofforward} below. In all this section we consider $(X,d)$ an Alexandrov space with non-negative curvature, and we fix a transport ray $\gamma:[0,a]\to X$ from the full-measure family of Step~0 below (see Section \ref{sec:proofforward}) and a compact subinterval $[\alpha,\beta]\subset(0,a)$, and construct a family of Petrunin-type transports between the normal cones, on a (shifted) dyadic set of times, enjoying \emph{both} properties: an exact cocycle, and the second-variation bound $(\mathrm{SV}_\perp)$ in the directions normal to the ray (Proposition~\ref{prop:cocycle}). The construction proceeds in two stages. First, a forward family $\tau$ of Petrunin-type transports with an exact cocycle is extracted along a fixed ultrafilter (Lemma~\ref{lem:taucocycle}); transcribing Petrunin's second-variation argument to this setting proves $(\mathrm{SV}_\perp)$, per pair, for the \emph{midpoint-factorised} variants $T_{s,t}$ built from $\tau$ and the reverse family $\sigma$ (Proposition~\ref{prop:SVmidpoint}). Second, the trivialisation $W$ is defined as the limit of compositions of elementary midpoint-factorised blocks: the cocycle then holds by associativity at every level, and $(\mathrm{SV}_\perp)$ is inherited through the composition lemma (Lemma~\ref{lem:SVcompose}). 

\medskip
\noindent\emph{Conventions.} Fix once and for all a nonprincipal ultrafilter $\omega$ on $\mathbb N$ and set $\varepsilon_j:=1/j$ (any fixed scale sequence would do). All limits below are $\omega$-limits; since every quantity involved lies in a compact set (closed balls in proper tangent cones), every $\omega$-limit exists, and no subsequence extraction occurs anywhere --- this is what makes a single coherent selection possible. For an interior point $p=\gamma(t)$ we write $L_p\subset T_pX$ for the normal cone $\{v\in T_pX:v\perp\dot\gamma\}$, with apex $o$, so that $T_pX=L_p\times\R\dot\gamma$ \cite[7.15]{BuragoGromovPerelman92}. For $\varphi\in[0,a)$ let $D_\varphi:=\bigl(\varphi+a\,\mathbb Z[\tfrac12]\bigr)\cap[\alpha,\beta]$ denote the shifted dyadic grid, $\mathbb Z[\tfrac12]$ being the dyadic rationals.

\begin{lem}[good phases]\label{lem:goodphase}
For a.e.\ $\varphi\in[0,a)$, every $t\in D_\varphi$ is a \emph{good time}: $\gamma(t)$ is a regular point of $X$ and lies in $\operatorname{Reg}\theta_t$, so that $T_{\gamma(t)}X\cong\R^\dd$ and $U_t=\hess_{\gamma(t)}\theta_t$ is defined.
\end{lem}

\begin{proof}
Bad times form a Lebesgue-null subset of $(0,a)$. Indeed, a time $t$ is bad iff $\gamma(t)\in N_t:=S_X\cup\bigl(X\setminus\operatorname{Reg}\theta_t\bigr)$, and each $N_t$ is $\mu_H$-null ($\mu_H(S_X)=0$ by \cite[Thm.~A]{OtsuShioya94}, and $\mu_H(X\setminus\operatorname{Reg}\theta_t)=0$ by Alexandrov's theorem in the DC charts \cite{Perelman_DC,otsu30differential}, $\theta_t$ being semiconcave). By Lemma~\ref{lem:nullrays} of Appendix~\ref{app:rays}, for $\mu_H$-a.e.\ transport ray $\gamma$ the set $\{t:\gamma(t)\in N_t\}$ is Lebesgue-null. For fixed $k\in\mathbb Z$, $m\in\mathbb N$ the shift $\varphi\mapsto\varphi+ka2^{-m}$ preserves Lebesgue-null sets; intersecting over the countably many $(k,m)$ gives the claim.
\end{proof}

Fix henceforth a good phase $\varphi$ and write $p^m_k:=\gamma(\varphi+ka2^{-m})$ for
the level-$m$ grid points lying in $[\alpha,\beta]$. For consecutive level-$m$ grid
points $f=p^m_k$, $g=p^m_{k+1}$ define, following Petrunin's construction \cite[\S1.4]{petrunin1998parallel},
\[
\pi^m_k:=\mathrm{pr}_{L}\circ\log_{g}\circ\exp_{f},
\]
where we recall that for all points $p_k^m=\gamma(s_k^m)$, the tangent space splits orthogonally as
\[
T_{\gamma(s_k^m)}X\ =\ \R\dot\gamma(s_k^m)\ \oplus\ L_{\gamma(s_k^m)},
\qquad L_{\gamma(s_k^m)}:=\dot\gamma(s_k^m)^{\perp},
\]
and $\mathrm{pr}_{L}: T_{p_k^m} X \to L_{p_k^m} $ is the projection onto the orthogonal part.
We define then the \emph{one-step map}
\[
\Pi^m_k(x):=\omega\text{-}\lim_j\ \frac{\pi^m_k(\varepsilon_j x)}{\varepsilon_j}\ \in
L_{p^m_{k+1}},\qquad x\in L_{p^m_k},
\]
and for grid points $s<t$ of level $\le m$, we define the \emph{level-$m$ composite}
\[
\tau^m_{s\to t}:=\Pi^m_{k_t-1}\circ\cdots\circ\Pi^m_{k_s},
\]
composing the level-$m$ one-step maps that subdivide $[s,t]$; and finally we define the \emph{transport}
\[
\tau_{s\to t}:=\omega\text{-}\lim_m\ \tau^m_{s\to t},
\]
pointwise on $L_{\gamma(s)}$. Reverse composites $\sigma^m_{t\to s}$, $\sigma_{t\to s}$
are defined in the same way with the roles of $f,g$ exchanged.

\begin{lem}[one-step maps]\label{lem:onestep}
Each $\Pi^m_k$ is well defined on all of $L_{p^m_k}$, fixes the apex $o$, and is noncontracting; and there is a constant $C$ such that for any $\varepsilon>0$ and $r\in T_{p^m_k}X$,
$$
\pi^m_k(\varepsilon r)/\varepsilon\in B_{C|r|}(o)
$$
and $C$ may be chosen uniformly over $m,k$, all steps lying in $\gamma|_{[\alpha,\beta]}$, whose distance to the endpoints of the ray is bounded below.
\end{lem}

\begin{proof}
By the Lemma of \cite[\S1.4]{petrunin1998parallel}, $\pi^m_k(\varepsilon r)/\varepsilon\in B_{C|r|}(o)$ with $C$ depending only on the distance from the step to the endpoints of $\gamma$; bounded subsets of the proper cone $L_{p^m_{k+1}}$ have unique $\omega$-limits, so $\Pi^m_k$ is defined everywhere. The apex is fixed since $\mathrm{pr}_L\log_g f=o$. Noncontractivity is the Lemma of \cite[\S1.5]{petrunin1998parallel}, which survives the replacement of subsequential limits by $\omega$-limits.
\end{proof}

\begin{lem}[norm growth]\label{lem:normgrowth}
There is $C=C(\alpha,\beta,\gamma)$ such that for all grid points $s<t$ of level $\le m$, with $N_m=(t-s)2^m/a$ the number of level-$m$ steps subdividing $[s,t]$,
\[
\bigl|\tau^m_{s\to t}(v)\bigr|\ \le\ \Bigl(1+\frac{C}{N_m}\Bigr)|v|,
\qquad v\in L_{\gamma(s)},
\]
and likewise for the reverse composites $\sigma^m_{t\to s}$.
\end{lem}

\begin{proof}
This is \cite[Lemma~1.8]{petrunin1998parallel}, whose constant depends only on the distance from the subdivided segment to the endpoints of the geodesic, hence is uniform here. 
\end{proof}

\begin{lem}[the limits are orthogonal maps]\label{lem:limitisometry}
For all grid points $s<t$ in $D_\varphi$, the map $\tau_{s\to t}$ is a surjective linear isometry $L_{\gamma(s)}\to L_{\gamma(t)}$, and likewise $\sigma_{t\to s}$.
\end{lem}

\begin{proof}
Write $\tau:=\tau_{s\to t}$, $\sigma:=\sigma_{t\to s}$. Both are noncontracting ($\omega$-limits of compositions of noncontracting maps, Lemma~\ref{lem:onestep}) and satisfy $|\tau v|\le\lim_m(1+C/N_m)|v|=|v|$ and $|\sigma w|\le|w|$ (Lemma~\ref{lem:normgrowth}). Hence $F:=\sigma\circ\tau$ is a noncontracting self-map of the compact ball $\clball_R(o)\subset L_{\gamma(s)}$ for every $R>0$, and is therefore a surjective isometry of that ball (the folklore lemma of \cite[\S1.2]{petrunin1998parallel}). From $|xy|=|F(x)F(y)|\ge|\tau x\,\tau y|\ge|xy|$ the map $\tau$ is distance-preserving, and surjectivity of $F$ and of $\tau\circ\sigma$ gives surjectivity of $\tau$. Since grid times are good (Lemma~\ref{lem:goodphase}), $L_{\gamma(s)}$ and $L_{\gamma(t)}$ are Euclidean cones $\cong\R^{\dd-1}$, and a surjective, apex-fixing isometry between Euclidean spaces is linear (Mazur--Ulam).
\end{proof}

\begin{lem}[composition lemma]\label{lem:graphlimit}
Let $A_m$ be noncontracting maps between fixed compact balls, converging pointwise
($\omega$) to a surjective isometry $A$. If $y_m\to y$ ($\omega$), then
$A_m(y_m)\to A(y)$ ($\omega$). Consequently
$\omega\text{-}\lim_m(A_m\circ B_m)=A\circ B$ pointwise whenever $B_m\to B$ pointwise
($\omega$).
\end{lem}

\begin{proof}
By compactness set $z:=\omega\text{-}\lim A_m(y_m)$. For every $x$ in the source,
$|A_m(y_m)\,A_m(x)|\ge|y_m\,x|$; taking $\omega$-limits,
$|z\,A(x)|\ge|y\,x|=|A(y)\,A(x)|$. Since $A$ is surjective, $A(x)$ ranges over the whole
target ball; choosing $x$ with $A(x)=z$ yields $0\ge|A(y)\,z|$, i.e.\ $z=A(y)$.
\end{proof}

\begin{lem}[exact cocycle on the grid]\label{lem:taucocycle}
For a.e.\ phase $\varphi$, the family $\{\tau_{s\to t}\}_{s<t\in D_\varphi}$ consists of
surjective linear isometries $L_{\gamma(s)}\to L_{\gamma(t)}$ satisfying the exact
cocycle identity
\begin{equation}\label{eq:taucocycle}
\tau_{r\to t}\circ\tau_{s\to r}=\tau_{s\to t},\qquad s<r<t\ \text{in }D_\varphi .
\end{equation}
\end{lem}

\begin{proof}
At every finite level $m\ge\operatorname{level}(s,r,t)$ the identity
$\tau^m_{s\to t}=\tau^m_{r\to t}\circ\tau^m_{s\to r}$ holds \emph{exactly}: both sides
are the composition of the same level-$m$ one-step maps. Pass to $\omega$-limits in $m$
using Lemma~\ref{lem:graphlimit}, whose hypotheses hold by
Lemma~\ref{lem:limitisometry}.
\end{proof}

The family $\tau$ is, by construction, of Petrunin's type: it arises from the
same comparison-configuration limits as \cite[Thm.~1.1]{petrunin1998parallel}. What
Theorem~1.1.B of \cite{petrunin1998parallel} provides, however, is the second-variation
bound for \emph{some} pairwise transport obtained through additional volume-correction
limits (\cite[\S\S1.12--1.16]{petrunin1998parallel}), not for a prescribed family. We transcribe the proof of \cite[\S\S1.12--1.16]{petrunin1998parallel} to the fixed-ultrafilter setting. Since at grid
times all tangent cones are Euclidean (Lemma~\ref{lem:goodphase}), this transcription allows to replace Perelman's
bi-Lipschitz neighbourhood argument and the isoperimetric machinery of \cite[\S\S1.11, 1.14--1.15]{petrunin1998parallel} by an elementary rigidity lemma and the Brunn--Minkowski inequality.
However, this proves the second-variation bound for the \emph{midpoint-factorised} isometries $T_{s,t}=(\sigma_{t\to u})^{-1}\circ\tau_{s\to u}$ ($u$ the midpoint of $s,t$), not for $\tau_{s\to t}$ itself. Contrary to what one might first hope, the cocycle \eqref{eq:taucocycle} does \emph{not} collapse $T_{s,t}$ to $\tau_{s\to t}$: the second factor of the factorisation is built from the \emph{reverse} one-step maps, and
\eqref{eq:taucocycle} relates only forward composites among themselves; the collapse
would amount to the inversion identity
$\sigma_{t\to s}\circ\tau_{s\to t}=\mathrm{Id}$, which the available estimates do not
control. We therefore do \emph{not} attempt to identify the two families. Instead, the
trivialisation of Proposition~\ref{prop:cocycle} below is assembled from the midpoint-factorised blocks themselves: compositions of elementary blocks at a fixed dyadic level satisfy the cocycle by associativity and $(\mathrm{SV}_\perp)$ by the
composition lemma, and both properties survive the limit in the level.

Throughout the sequel of this subsection, a statement of the form ``$A_j\le B_j+o(\varepsilon_j^2)$
for $\omega$-almost all $j$'' means: \emph{for every $\delta>0$ the set
$\{j:A_j\le B_j+\delta\varepsilon_j^2\}$ belongs to $\omega$}. 

Given grid times $s<t$ in $D_\varphi$ at distance $\ell:=d(\gamma(s),\gamma(t))$ and a surjective linear isometry $A:L_{\gamma(s)}\to L_{\gamma(t)}$, we say that $A$ \emph{satisfies $(\mathrm{SV}_\perp)$ for the pair $(s,t)$} if for all fixed $x\in L_{\gamma(s)}$, $y\in L_{\gamma(t)}$, and for $\omega$-almost all $j$,
\begin{equation}\label{SVperp}
\bigl|\exp_{\gamma(s)}(\varepsilon_j x)\ \exp_{\gamma(t)}(\varepsilon_j y)\bigr|\ \le\ \ell+\frac{|x-A^{-1}(y)|^2}{2\ell}\,\varepsilon_j^2+o(\varepsilon_j^2)
\end{equation}
where $\exp$ is taken along quasigeodesics, as in \cite[\S1.0]{petrunin1998parallel}.

\begin{lem}[rigidity of almost-isometries of $\R^d$]\label{lem:rigidity}
Let $d\ge1$, $\delta\in(0,1]$, $R>0$, and let $A:\R^d\to\R^d$ satisfy $A(0)=0$, $|A(v)-A(w)|\ge|v-w|$ for all $v,w$, and $|A(v)|\le(1+\delta)|v|$ for all $v$. Then there is a linear isometry $I$ of $\R^d$ with
\[
|A(v)-I(v)|\ \le\ C_d\,\delta^{1/4}R\qquad\text{for all }|v|\le R,
\]
with $C_d$ depending only on $d$.
\end{lem}

\begin{proof}
Since $|A(v)|\le(1+\delta)|v|\le2R$ on $\clball_R(0)$, the conclusion is trivial (for any choice of $I$, say $I=\mathrm{Id}$, and $C_d$ large) unless $\delta\le\delta_d$ for a dimensional constant fixed at the end; assume this. First, for $|v|\le R$,
\begin{align*}
|A(v)+A(-v)|^2 &=2|A(v)|^2+2|A(-v)|^2-|A(v)-A(-v)|^2\\
&\le4(1+\delta)^2|v|^2-4|v|^2\\
&\le12\,\delta|v|^2,
\end{align*}
so $|A(v)+A(-v)|\le4\sqrt\delta\,R$. Next, for $|v|,|w|\le R$, noncontractivity gives $\<A(v),A(w)\>\le\tfrac12\bigl(|A(v)|^2+|A(w)|^2-|v-w|^2\bigr)\le\<v,w\>+3\delta R^2$, and, writing $A(w)=\bigl(A(w)+A(-w)\bigr)-A(-w)$,
\[
\<A(v),A(w)\>\ \ge\ -|A(v)|\,\bigl|A(w)+A(-w)\bigr|-\<A(v),A(-w)\> \ \ge\ \<v,w\>-12\sqrt\delta\,R^2 ,
\]
using the previous two estimates.
Hence $\bigl|\<A(v),A(w)\>-\<v,w\>\bigr|\le12\sqrt\delta\,R^2$ on $\clball_R(0)$. Let $(e_i)$ be an orthonormal basis and $u_i:=A(Re_i)/R$; the Gram matrix $G=(\<u_i,u_j\>)$ satisfies $\|G-\mathrm{Id}\|\le12d\sqrt\delta$, so for $\delta\le\delta_d$ small the L\"owdin orthonormalisation $\bar e_i:=\sum_k(G^{-1/2})_{ki}u_k$ is defined and $|\bar e_i-u_i|\le C'_d\sqrt\delta$. Let $I$ be the linear isometry with $I(e_i)=\bar e_i$. For $|v|\le R$, $v=\sum v_ie_i$,
\[
\<A(v),I(v)\>=\sum_iv_i\<A(v),u_i\>+\sum_iv_i\<A(v),\bar e_i-u_i\> \ \ge\ |v|^2-C''_d\sqrt\delta\,R^2 ,
\]
since $\<A(v),u_i\>=\<A(v),A(Re_i)\>/R\ge\<v,e_i\>-12\sqrt\delta R$ and $\sum_i|v_i|\le\sqrt d\,R$. Therefore
\[
|A(v)-I(v)|^2=|A(v)|^2-2\<A(v),I(v)\>+|v|^2\le3\delta R^2+2C''_d\sqrt\delta R^2 \le C_d^2\sqrt\delta\,R^2 . \qedhere
\]
\end{proof}

The hypothesis that $(X, d)$ is $\curv\ge0$ that we assume throughout this section is used and is necessary for the next Proposition \ref{prop:SVmidpoint}.

\begin{prop}[second variation for the midpoint-factorised
maps]\label{prop:SVmidpoint}
Let $s<t$ in $D_\varphi$, let $u:=\tfrac{s+t}2$ (a grid time, the grid being dyadic),
and set
\[
T_{s,t}\ :=\ (\sigma_{t\to u})^{-1}\circ\tau_{s\to u}\ :\ L_{\gamma(s)}\to
L_{\gamma(t)},
\]
a surjective linear isometry by Lemma~\ref{lem:limitisometry}. Then:
\begin{enumerate}
\item[(i)] $T_{s,t}$ satisfies $(\mathrm{SV}_\perp)$ for the pair $(s,t)$;
\item[(ii)] consequently
$\widehat T_{s,t}:=T_{s,t}\oplus(\dot\gamma(s)\mapsto\dot\gamma(t))$ satisfies the
model second-variation bound \emph{restricted to normal vectors}: there are
$\tilde p,\tilde q\in\mathbb E^\dd$ with $|\tilde p\tilde q|=\ell$ and linear isometries
$\iota_s:T_{\gamma(s)}X\to T_{\tilde p}\mathbb E^\dd$,
$\iota_t:T_{\gamma(t)}X\to T_{\tilde q}\mathbb E^\dd$ with
$\iota_t\circ\widehat T_{s,t}=\tilde\tau\circ\iota_s$, $\tilde\tau$ the Euclidean parallel translation, such that for all fixed $v\in L_{\gamma(s)}$,
$w\in L_{\gamma(t)}$,
for $\omega$-almost all $j$, 
\[
\bigl|\exp_{\gamma(s)}(\varepsilon_j v)\ \exp_{\gamma(t)}(\varepsilon_j w)\bigr|\ \le\
\bigl|\exp_{\tilde p}\,\iota_s(\varepsilon_j v)\ \exp_{\tilde q}\,\iota_t(\varepsilon_j w)\bigr|
+o(\varepsilon_j^2).
\]
\end{enumerate}
\end{prop}

\begin{proof}
If $X$ has dimension $\dd=1$ the normal cones are trivial and there is nothing to prove; assume $\dd\ge2$
and write $d:=\dd-1$, so that all normal cones at grid times are Euclidean
$\cong\R^{d}$ (Lemma~\ref{lem:goodphase}). Write $p:=\gamma(s)$, $q:=\gamma(t)$,
$r:=\gamma(u)$; $r$ is the midpoint of the segment $pq$ of the ray. Fix
$x\in L_p$, $y\in L_q$, set $\hat y:=T_{s,t}^{-1}(y)\in L_p$, and fix a rational
$\lambda>|x-\hat y|/\ell$. We will prove
\begin{equation}\label{eq:SVlambda}
\bigl|\exp_p(\varepsilon_j x)\,\exp_q(\varepsilon_j y)\bigr|\ \le\
\ell\Bigl(1+\frac{\lambda^2}2\,\varepsilon_j^2\Bigr)+o(\varepsilon_j^2)
\qquad\text{for $\omega$-almost all }j ;
\end{equation}
since for every $\delta>0$ one may choose rational $\lambda$ with
$\lambda^2\ell/2\le|x-\hat y|^2/(2\ell)+\delta/2$, the family of statements
\eqref{eq:SVlambda} (one for each $\lambda$, the infimum over $\lambda$ being taken
\emph{outside} the $\omega$-quantifier) yields $(\mathrm{SV}_\perp)$.

For $m$ at least the level of $s,t$, let $p_i:=\gamma(s+ia2^{-m})$,
$0\le i\le N$, be the level-$m$ grid points subdividing $[s,u]$, so $p_0=p$,
$p_N=r$, $N=N_m:=(u-s)2^m/a$, with step length $\delta_m:=\ell/(2N)$; let
$q_i:=\gamma(t-ia2^{-m})$ be the points subdividing $[t,u]$ in the reverse direction,
$q_0=q$, $q_N=r$. Write $\Pi_i:L_{p_i}\to L_{p_{i+1}}$ for the level-$m$ one-step maps
and $\Pi'_i:L_{q_i}\to L_{q_{i+1}}$ for the reverse ones, so that
$\tau^m_{s\to u}=\Pi_{N-1}\circ\cdots\circ\Pi_0$ and
$\sigma^m_{t\to u}=\Pi'_{N-1}\circ\cdots\circ\Pi'_0$. All statements below are proved
for the forward chain and hold verbatim for the reverse chain, Petrunin's one-step
construction being symmetric in the two endpoints of a step.

\smallskip\noindent\emph{(a) One-step inequality} (transcription of
\cite[\S1.6]{petrunin1998parallel}). Let $f:=p_i$, $g:=p_{i+1}$, $\Pi:=\Pi_i$,
$\pi:=\mathrm{pr}_L\circ\log_g\circ\exp_f$. For all $x_*\in L_f$, $x'\in L_g$ and all
$\varepsilon>0$, \emph{before any limit is taken},
\begin{equation}\label{eq:onesteppre}
\bigl|\exp_f(\varepsilon x_*)\,\exp_g(\varepsilon x')\bigr|^2\ \le\
|fg|^2+\varepsilon^2|x_*|^2-|\pi(\varepsilon x_*)|^2
+\bigl|\pi(\varepsilon x_*)-\varepsilon x'\bigr|^2 .
\end{equation}
Indeed: $|\varepsilon x'\ \log_g\exp_f(\varepsilon x_*)|\ge
|\exp_g(\varepsilon x')\,\exp_f(\varepsilon x_*)|$ by the defining property of
$\exp,\log$ in nonnegative curvature ; in the splitting $T_gX=L_g\times\R\dot\gamma$ the squared distance from $\varepsilon x'\in
L_g$ to $w:=\log_g\exp_f(\varepsilon x_*)$ equals
$\bigl(|w|^2-|\mathrm{pr}_Lw|^2\bigr)+|\mathrm{pr}_Lw-\varepsilon x'|^2$; and
$|w|=|g\,\exp_f(\varepsilon x_*)|\le\sqrt{|fg|^2+\varepsilon^2|x_*|^2}$ because
$x_*\perp\dot\gamma$ and $\rho\mapsto|g\,c(\rho)|^2-\rho^2$ is concave along the
quasigeodesic $c(\rho)=\exp_f(\rho x_*/|x_*|)$ in $\curv\ge0$, with vanishing initial
derivative. Now take $\varepsilon=\varepsilon_j$: by Lemma~\ref{lem:onestep},
$\pi(\varepsilon_jx_*)/\varepsilon_j\to_\omega\Pi(x_*)$ in a compact ball, whence for the \emph{fixed} pair $(x_*,x')$, and for $\omega$-a.a. $j$,
\begin{equation}\label{eq:onestepomega}
\bigl|\exp_f(\varepsilon_j x_*)\,\exp_g(\varepsilon_j x')\bigr|^2\le
|fg|^2+\bigl(|x_*|^2-|\Pi(x_*)|^2+|\Pi(x_*)-x'|^2\bigr)\varepsilon_j^2
+o(\varepsilon_j^2).
\end{equation}
In particular, if $|\Pi(x_*)-x'|\le\lambda\delta_m$ then, using
$|\Pi(x_*)|\ge|x_*|$ (Lemma~\ref{lem:onestep}) and $\sqrt{1+c}\le1+c/2$,
\begin{equation}\label{eq:onestepball}
\bigl|\exp_f(\varepsilon_j x_*)\,\exp_g(\varepsilon_j x')\bigr|\ \le\
\delta_m\Bigl(1+\frac{\lambda^2}2\varepsilon_j^2\Bigr)+o(\varepsilon_j^2)
\quad\text{for $\omega$-a.a.\ }j.
\end{equation}

\smallskip\noindent\emph{(b) Chains.} Call a \emph{$\lambda$-chain from $x$} a finite
sequence $x_0=x$, $x_i\in L_{p_i}$, with $|x_{i+1}-\Pi_i(x_i)|\le\lambda\delta_m$ for
all $i$. A $\lambda$-chain consists of finitely many fixed vectors, so the finitely
many statements \eqref{eq:onestepball} along it may be combined: by the triangle
inequality through the points $\exp_{p_i}(\varepsilon_jx_i)$, every full $\lambda$-chain
$x_0=x,\dots,x_N=:x'\in L_r$ satisfies
\begin{equation}\label{eq:chainineq}
\bigl|\exp_p(\varepsilon_jx)\,\exp_r(\varepsilon_jx')\bigr|\ \le\
\frac\ell2\Bigl(1+\frac{\lambda^2}2\varepsilon_j^2\Bigr)+o(\varepsilon_j^2)
\qquad\text{for $\omega$-a.a.\ }j,
\end{equation}
the analogue of Inequality (\#\#\#) in \cite[\S1.13]{petrunin1998parallel}. Define the propagated
sets
\[
\Sigma_m(x):=\Phi_{N-1}\circ\cdots\circ\Phi_0(\{x\})\subset L_r,\qquad
\Phi_i:=B_{\lambda\delta_m}\circ\Pi_i ,
\]
$B_\rho(Y)$ denoting the open $\rho$-neighbourhood; every point of $\Sigma_m(x)$ is the
endpoint of a $\lambda$-chain from $x$, and $\Sigma_m(x)$ is open ($B_\rho(Y)$ is open
for any $Y$). Let $\Sigma'_m(y)\subset L_r$ be the set propagated from $y\in L_q$ along
the reverse chain.

\smallskip\noindent\emph{(c) Geometry of the propagated sets.} We prove three claims.

\emph{(c1) Norm bound.} There are $m_0$ and $R'=R'(|x|,\lambda,\ell)$ such that for all
$m\ge m_0$, every $\lambda$-chain from $x$ satisfies $\max_i|x_i|\le R'$. This is the
transcription of the first part of the Lemma in \cite[\S1.13a]{petrunin1998parallel},
which in turn follows \cite[\S1.8]{petrunin1998parallel}. Fix auxiliary points
$a^-:=\gamma(\alpha/2)$ and $a^+:=\gamma(\tfrac{\beta+a}2)$ on the ray, on either side
of $[\alpha,\beta]$, at distance $\ge\rho_0>0$ from all grid points; set
$C_0:=1/(2\rho_0)$. Let $(x_i)_{i\le k}$ be a $\lambda$-chain stopped at the index $k$
realising $M:=\max_i|x_i|$. Exactly as in (a), the quasigeodesic comparison gives the
\emph{pre-limit} endpoint bounds
$|a^-\,\exp_{p_0}(\varepsilon x_0)|\le|a^-p_0|+C_0|x_0|^2\varepsilon^2$ and
$|\exp_{p_k}(\varepsilon x_k)\,a^+|\le|p_ka^+|+C_0|x_k|^2\varepsilon^2$. Combining them
with \eqref{eq:onestepomega} along the chain via the triangle inequality, and using the
collinearity $|a^-p_0|+k\delta_m+|p_ka^+|=|a^-a^+|$, we obtain after subtracting
$|a^-a^+|$, dividing by $\varepsilon_j^2$ and letting $j\to\omega$: 
\[
0\ \le\ C_0\bigl(|x_0|^2+M^2\bigr)
+\frac1{2\delta_m}\sum_{i<k}\Bigl(|x_i|^2-|\Pi_i(x_i)|^2+\lambda^2\delta_m^2\Bigr).
\]
Writing $|x_i|^2-|\Pi_i(x_i)|^2
=\bigl(|x_i|^2-|x_{i+1}|^2\bigr)+\bigl(|x_{i+1}|^2-|\Pi_i(x_i)|^2\bigr)$, telescoping
the first summands, and estimating
$|x_{i+1}|^2-|\Pi_i(x_i)|^2\le2\lambda\delta_m|\Pi_i(x_i)|+\lambda^2\delta_m^2
\le2\lambda\delta_m(M+\lambda\delta_m)+\lambda^2\delta_m^2$ together with
$k\delta_m\le\ell/2$, this yields
\[
M^2\ \le\ 2\delta_mC_0\bigl(|x_0|^2+M^2\bigr)+|x_0|^2+\lambda\ell M
+2\lambda^2\ell\,\delta_m(1+\lambda) .
\]
For $\delta_m\le\delta_0(C_0)$ the quadratic inequality gives
$M\le C_1\bigl(1+\lambda\bigr)\bigl(1+\ell\bigr)\bigl(1+|x_0|\bigr)=:R'$, uniformly
over $\lambda$-chains; the bound for the reverse chain, with $|x|$ replaced by $|y|$,
defines $R''$. Set $\bar R:=\max(R',R'')$.

\emph{(c2) Containment.} For every $v\in L_p$ and every $w\in\Sigma_m(x)$,
noncontractivity of the $\Pi_i$ propagates the lower bound
\[
\bigl|\tau^m_{s\to u}(v)\ w\bigr|\ \ge\ |v-x|-N\lambda\delta_m=|v-x|-\lambda\ell/2 ,
\]
by induction along the chain ($|\Pi_i(v_i)\,w'|\ge|\Pi_i(v_i)\,\Pi_i(w_i)|
-\lambda\delta_m\ge|v_i-w_i|-\lambda\delta_m$ for $w'\in B_{\lambda\delta_m}(\Pi_i(w_i))$).
By Lemma~\ref{lem:normgrowth} the composite $\tau^m_{s\to u}$ is noncontracting, fixes
the apex and has norm $\le1+C/N_m$; Lemma~\ref{lem:rigidity} (with
$\delta=C/N_m\to0$, $R=\bar R$) provides linear isometries $\Upsilon_m$ of
$L_p\to L_r$ with
\[
\bigl|\tau^m_{s\to u}(v)-\Upsilon_m(v)\bigr|\le\eta_m:=C_d\,(C/N_m)^{1/4}\,\bar R
\ \xrightarrow[m\to\infty]{}\ 0\qquad(|v|\le\bar R).
\]
Given $w\in\Sigma_m(x)$, the norm bound (c1) gives $|w|\le\bar R$; applying the
displayed lower bound to $v:=\Upsilon_m^{-1}(w)$ (also of norm $\le\bar R$) and the
rigidity estimate to the same $v$ yields $|v-x|\le\lambda\ell/2+\eta_m$, i.e.
\[
\Sigma_m(x)\ \subset\ \clball_{\lambda\ell/2+\eta_m}\bigl(\Upsilon_m(x)\bigr).
\]
This replaces the Lemma of \cite[\S1.13a]{petrunin1998parallel}.
Likewise $\Sigma'_m(y)\subset\clball_{\lambda\ell/2+\eta'_m}(\Upsilon'_m(y))$ with isometries
$\Upsilon'_m\approx\sigma^m_{t\to u}$.

\emph{(c3) Volume.} $\vol\bigl(\Sigma_m(x)\bigr)\ge\vol\bigl(B_{\lambda\ell/2}\bigr)$,
where $\vol$ is Lebesgue measure on $L_r\cong\R^{d}$ and $B_\rho$ a Euclidean
$\rho$-ball. Indeed, denote by $\vol^*$ outer Lebesgue measure. If
$S\subset L_{p_i}\cong\R^d$ is bounded and nonempty then: (i)
$\vol^*\bigl(\Pi_i(S)\bigr)\ge\vol^*(S)$, because $\Pi_i$ is injective with
$1$-Lipschitz inverse on its image, which extends to a $1$-Lipschitz self-map of
$\R^d$ by Kirszbraun, and Lipschitz maps with constant $1$ do not increase outer
measure; (ii) $B_\rho(S)\supset\overline S+B^\circ_\rho$ (distance to a set equals
distance to its closure), so the Brunn--Minkowski inequality applied to the compact set
$\overline S$ and the open ball gives
$\vol\bigl(B_\rho(S)\bigr)^{1/d}\ge\vol^*(S)^{1/d}+\vol(B_\rho)^{1/d}$. Starting from
the singleton $\{x\}$ and applying (i)--(ii) along the $N$ steps,
$\vol(\Sigma_m(x))^{1/d}\ge N\lambda\delta_m\,\vol(B_1)^{1/d}
=\vol(B_{\lambda\ell/2})^{1/d}$. (This replaces the half-volume Claim of
\cite[\S1.15]{petrunin1998parallel}; no bi-Lipschitz neighbourhood $U_\delta$ is
needed, the cones being Euclidean.) The same bound holds for $\Sigma'_m(y)$.

\smallskip\noindent\emph{(d) The two propagated sets meet.} Let
$f:=\tau_{s\to u}$ and $g:=\sigma_{t\to u}$, so that $g\circ T_{s,t}=\tau_{s\to u}$ by
definition of $T_{s,t}$, whence
\[
g(y)=g\bigl(T_{s,t}\hat y\bigr)=f(\hat y),\qquad
\bigl|f(x)-f(\hat y)\bigr|=|x-\hat y|<\lambda\ell .
\]
Set $\rho:=\lambda\ell/2$ and $\zeta:=\bigl(\lambda\ell-|x-\hat y|\bigr)/2>0$. The
centres converge along $\omega$:
$|\Upsilon_m(x)-f(x)|\le\eta_m+|\tau^m_{s\to u}(x)-f(x)|\to_\omega0$ and
$|\Upsilon'_m(y)-g(y)|\to_\omega0$, so the set
\[
W:=\bigl\{m:\ |\Upsilon_m(x)-\Upsilon'_m(y)|\le2\rho-\zeta\ \text{and}\
\eta_m\vee\eta'_m\le\bar\eta_m\bigr\}\in\omega
\]
(for any fixed sequence $\bar\eta_m\downarrow0$ dominating $\eta_m,\eta'_m$, which
exists since $\eta_m,\eta'_m\to0$ deterministically). Suppose
$\Sigma_m(x)\cap\Sigma'_m(y)=\emptyset$ for some $m\in W$. Both sets are open, hence
measurable, so their volumes add; by (c2) and (c3),
\[
2\vol(B_\rho)\ \le\ \vol\Bigl(\clball_{\rho+\bar\eta_m}\bigl(\Upsilon_m(x)\bigr)\cup
\clball_{\rho+\bar\eta_m}\bigl(\Upsilon'_m(y)\bigr)\Bigr)
\ \le\ 2\vol(B_{\rho+\bar\eta_m})-\vol\bigl(B_{\zeta/2}\bigr),
\]
the last step because $\clball_{\rho+\bar\eta_m}(\Upsilon_m(x))\cap\clball_{\rho+\bar\eta_m}(\Upsilon'_m(y))$ contains the ball of radius $\rho+\bar\eta_m-\tfrac12|\Upsilon_m(x)-\Upsilon'_m(y)| \ge\zeta/2$ about the midpoint of the centres. Since $2\vol(B_{\rho+\bar\eta_m})-2\vol(B_\rho)\to0$ while $\vol(B_{\zeta/2})>0$ is fixed, this fails for all large $m$; as $W\in\omega$ is infinite, there is $m$ with
\[
\Sigma_m(x)\cap\Sigma'_m(y)\ \neq\ \emptyset .
\]

\smallskip\noindent\emph{(e) Conclusion.} Fix such an $m$ and
$z'\in\Sigma_m(x)\cap\Sigma'_m(y)$. The point $z'$ is the endpoint of a forward
$\lambda$-chain from $x$ and of a reverse $\lambda$-chain from $y$; by
\eqref{eq:chainineq} for each and the triangle inequality through
$\exp_r(\varepsilon_jz')$,
\[
\bigl|\exp_p(\varepsilon_jx)\,\exp_q(\varepsilon_jy)\bigr|\ \le\
\ell\Bigl(1+\frac{\lambda^2}2\varepsilon_j^2\Bigr)+o(\varepsilon_j^2)
\qquad\text{for $\omega$-a.a.\ }j,
\]
which is \eqref{eq:SVlambda}. This proves (i).

For (ii), choose $\tilde p,\tilde q\in\mathbb E^\dd$ with $|\tilde p\tilde q|=\ell$,
let $e_1$ be the unit vector from $\tilde p$ to $\tilde q$, let
$\iota_s:T_{\gamma(s)}X\to T_{\tilde p}\mathbb E^\dd$ be any linear isometry taking the
direction of $\dot\gamma(s)$ to $e_1$, and set
$\iota_t:=\tilde\tau\circ\iota_s\circ\widehat T_{s,t}^{\,-1}$, so that
$\iota_t\circ\widehat T_{s,t}=\tilde\tau\circ\iota_s$ and $\iota_t$ takes the direction
of $\dot\gamma(t)$ to $e_1$. For normal $v\in L_{\gamma(s)}$, $w\in L_{\gamma(t)}$ the
model distance is
$\bigl|\exp_{\tilde p}\iota_s(\varepsilon v)\ \exp_{\tilde q}\iota_t(\varepsilon w)\bigr|
=\sqrt{\ell^2+\varepsilon^2|v-T_{s,t}^{-1}w|^2}
\ge\ell+\frac{|v-T_{s,t}^{-1}w|^2}{2\ell}\varepsilon^2-O(\varepsilon^4)$, so
$(\mathrm{SV}_\perp)$ for $T_{s,t}$ is precisely the displayed model bound for
$\widehat T_{s,t}$ on normal vectors, with $\varepsilon=\varepsilon_j=1/j$
and the same ``$\omega$-almost all'' quantifier.
\end{proof}

\begin{lem}[$(\mathrm{SV}_\perp)$ composes along concatenations]\label{lem:SVcompose}
Let $s<u<t$ in $D_\varphi$, $\ell_1:=d(\gamma(s),\gamma(u))$,
$\ell_2:=d(\gamma(u),\gamma(t))$, $\ell=\ell_1+\ell_2$. If
$A:L_{\gamma(s)}\to L_{\gamma(u)}$ satisfies $(\mathrm{SV}_\perp)$ for $(s,u)$ and
$B:L_{\gamma(u)}\to L_{\gamma(t)}$ satisfies $(\mathrm{SV}_\perp)$ for $(u,t)$, then
$B\circ A$ satisfies $(\mathrm{SV}_\perp)$ for $(s,t)$.
\end{lem}

\begin{proof}
Fix $x\in L_{\gamma(s)}$, $y\in L_{\gamma(t)}$, set $\hat y:=(B\circ A)^{-1}(y)$ and
let $z:=x+\frac{\ell_1}\ell(\hat y-x)$ be the proportional split of the segment
$[x,\hat y]$ in the Euclidean space $L_{\gamma(s)}$, so that
$|x-z|=\frac{\ell_1}\ell|x-\hat y|$ and $|z-\hat y|=\frac{\ell_2}\ell|x-\hat y|$.
Apply $(\mathrm{SV}_\perp)$ for $(s,u)$ to the pair $(x,A(z))$ and for $(u,t)$ to the
pair $(A(z),y)$; note $A^{-1}(A(z))=z$ and $B^{-1}(y)=A(\hat y)$, so
$|A(z)-B^{-1}(y)|=|z-\hat y|$. The triangle inequality through
$\exp_{\gamma(u)}(\varepsilon_jA(z))$ gives, for $\omega$-a.a.\ $j$ (intersection of
two $\omega$-large sets),
\begin{align*}
\bigl|\exp_{\gamma(s)}(\varepsilon_jx)\,\exp_{\gamma(t)}(\varepsilon_jy)\bigr|&\le
\ell+\frac{\varepsilon_j^2}2\Bigl[\frac{|x-z|^2}{\ell_1}+\frac{|z-\hat y|^2}{\ell_2}
\Bigr]+o(\varepsilon_j^2)\\
&=\ell+\frac{|x-\hat y|^2}{2\ell}\,\varepsilon_j^2+o(\varepsilon_j^2),
\end{align*}
the proportional split being the equality case of Cauchy--Schwarz:
$\frac{(\ell_1/\ell)^2}{\ell_1}+\frac{(\ell_2/\ell)^2}{\ell_2}=\frac1\ell$.
\end{proof}

We can now state the main result of this section.

\begin{prop}[the trivialisation: exact cocycle and $(\mathrm{SV}_\perp)$]\label{prop:cocycle}
For grid times $s<t$ in $D_\varphi$ and $m\ge\operatorname{level}(s,t)$, let
$s=s_0<s_1<\dots<s_K=t$ be the level-$m$ grid points subdividing $[s,t]$, and set
\[
W^m_{s\to t}\ :=\ T_{s_{K-1},s_K}\circ\cdots\circ T_{s_0,s_1},
\qquad
W_{s\to t}\ :=\ \omega\text{-}\lim_m\ W^m_{s\to t},
\]
where the $T_{\cdot,\cdot}$ are the midpoint-factorised maps of
Proposition~\ref{prop:SVmidpoint}. Then the family $\{W_{s\to t}\}_{s<t\in D_\varphi}$
consists of surjective linear isometries $L_{\gamma(s)}\to L_{\gamma(t)}$ which:
\begin{enumerate}
\item[(i)] satisfy the exact cocycle identity
\begin{equation}\label{eq:cocycle}
W_{r\to t}\circ W_{s\to r}=W_{s\to t},\qquad s<r<t\ \text{in }D_\varphi ;
\end{equation}
\item[(ii)] satisfy $(\mathrm{SV}_\perp)$ for every pair $(s,t)$;
\item[(iii)] extend, by
$\widehat W_{s\to t}:=W_{s\to t}\oplus\bigl(\dot\gamma(s)\mapsto\dot\gamma(t)\bigr)$
(transport rays have constant speed, so this is isometric), to splitting-preserving linear isometries $T_{\gamma(s)}X\to T_{\gamma(t)}X$ with the same cocycle property, satisfying the model second-variation bound for normal vectors: with $\tilde p,\tilde q,\iota_s,\iota_t$ as in part~(ii) of Proposition~\ref{prop:SVmidpoint}, now relative to $\widehat W_{s\to t}$, for $\omega$-almost all $j$: for all fixed $v\in L_{\gamma(s)}$, and $w\in L_{\gamma(t)}$,
\[
\bigl|\exp_{\gamma(s)}(\varepsilon_j v)\ \exp_{\gamma(t)}(\varepsilon_j w)\bigr|\ \le\
\bigl|\exp_{\tilde p}\,\iota_s(\varepsilon_j v)\ \exp_{\tilde q}\,\iota_t(\varepsilon_j w)\bigr|
+o(\varepsilon_j^2).
\]
\end{enumerate}
\end{prop}

\begin{proof}
Each $W^m_{s\to t}$ is a surjective linear isometry $L_{\gamma(s)}\to L_{\gamma(t)}$,
being a composition of such maps (Proposition~\ref{prop:SVmidpoint},
Lemma~\ref{lem:limitisometry}). Grid-time normal cones are Euclidean
(Lemma~\ref{lem:goodphase}); fixing orthonormal bases once and for all, regard each
$W^m_{s\to t}$ as an element of the compact group $O(\dd-1)$. The $\omega$-limit in $m$
then exists entrywise in $O(\dd-1)$, and $W_{s\to t}$ is again a surjective linear
isometry.

(i) At every finite level $m\ge\operatorname{level}(s,r,t)$ the identity
$W^m_{r\to t}\circ W^m_{s\to r}=W^m_{s\to t}$ holds \emph{exactly}: $r$ is a grid point
of level $\le m$, so the level-$m$ subdivision of $[s,t]$ is the concatenation of those
of $[s,r]$ and $[r,t]$, and both sides are the composition, in order, of the same
elementary blocks $T_{s_i,s_{i+1}}$. Matrix multiplication is jointly continuous, so
the identity passes to the $\omega$-limit in $m$.

(ii) Fix $m\ge\operatorname{level}(s,t)$. Each elementary block $T_{s_i,s_{i+1}}$
satisfies $(\mathrm{SV}_\perp)$ for its pair (Proposition~\ref{prop:SVmidpoint}(i));
Lemma~\ref{lem:SVcompose}, iterated over the $K-1$ concatenations, shows that
$W^m_{s\to t}$ satisfies $(\mathrm{SV}_\perp)$ for the pair $(s,t)$, \emph{for every
such $m$}. To pass to the $\omega$-limit in $m$, note that the level enters the bound
only through the constant $|x-(W^m_{s\to t})^{-1}(y)|^2$, which is continuous in the
transport. Precisely: fix $x\in L_{\gamma(s)}$, $y\in L_{\gamma(t)}$ and a rational
$\lambda>|x-W_{s\to t}^{-1}(y)|/\ell$. Inversion is continuous on $O(\dd-1)$, so
$(W^m_{s\to t})^{-1}(y)\to_\omega W_{s\to t}^{-1}(y)$, and the set
$\{m:\ |x-(W^m_{s\to t})^{-1}(y)|<\lambda\ell\}$ belongs to $\omega$; in particular it
is nonempty. Fix a \emph{single} $m_*$ in it. Then $(\mathrm{SV}_\perp)$ for
$W^{m_*}_{s\to t}$, applied to the fixed pair $(x,y)$, gives
\begin{align*}
\bigl|\exp_{\gamma(s)}(\varepsilon_j x)\,\exp_{\gamma(t)}(\varepsilon_j y)\bigr|\ &\le\
\ell+\frac{\bigl|x-(W^{m_*}_{s\to t})^{-1}(y)\bigr|^2}{2\ell}\,\varepsilon_j^2
+o(\varepsilon_j^2)\\
& \le\ \ell+\frac{\lambda^2\ell}{2}\,\varepsilon_j^2+o(\varepsilon_j^2)
\end{align*}
for $\omega$-almost all $j$. As in the proof of Proposition~\ref{prop:SVmidpoint}, the
infimum over rational $\lambda$ is taken \emph{outside} the $\omega$-quantifier in $j$
(and so is the choice of $m_*$, one for each $\lambda$); letting
$\lambda\downarrow|x-W_{s\to t}^{-1}(y)|/\ell$ along the rationals yields
$(\mathrm{SV}_\perp)$ for $W_{s\to t}$.

(iii) The extension is isometric and splitting-preserving by construction, and the
cocycle for $\widehat W$ follows from (i), the tangential component
$\dot\gamma(s)\mapsto\dot\gamma(t)$ composing trivially. The packaging of the
normal-vector model bound is verbatim that of part~(ii) of
Proposition~\ref{prop:SVmidpoint}, with $T_{s,t}$ replaced by $W_{s\to t}$.
\end{proof}

\begin{rem}[the full entropy tensor under Property (SV)]\label{rem:fullSV}
For completeness, we record the full-tangent form of the second-variation property, but leave it as an open question.\\
\medskip
\textbf{Open question.}
For every pair $s<t$ in $D_\varphi$ there are $\tilde p,\tilde q\in\mathbb E^\dd$ with
$|\tilde p\tilde q|=d(\gamma(s),\gamma(t))$ and linear isometries
$\iota_s:T_{\gamma(s)}X\to T_{\tilde p}\mathbb E^\dd$,
$\iota_t:T_{\gamma(t)}X\to T_{\tilde q}\mathbb E^\dd$, with
$\iota_t\circ\widehat W_{s\to t}=\tilde\tau\circ\iota_s$ where $\tilde\tau$ is the
Euclidean parallel translation, such that for all fixed $v\in T_{\gamma(s)}X$ and
$w\in T_{\gamma(t)}X$, and for $\omega$-almost all $n$, 
\[
\bigl|\exp_{\gamma(s)}(\tfrac1n v)\ \exp_{\gamma(t)}(\tfrac1n w)\bigr|\ \le\
\bigl|\exp_{\tilde p}\,\iota_s(\tfrac1n v)\ \exp_{\tilde q}\,\iota_t(\tfrac1n w)\bigr|
+o(n^{-2}).
\]
Proposition~\ref{prop:cocycle}(iii) is exactly the restriction of Property (SV) to normal data; what is missing is the normal-tangent mixed case. Property (SV) is the missing ingredient to upgrade Theorem \ref{thm:aleximpliesdisplacementconvex} to the matrix displacement convexity of the full entropy tensor.
\end{rem}

We close this subsection with the following elementary lemma that produces the canonical extension required by Definition~\ref{def:alexentropytensor}.

\begin{lem}[monotone extension]\label{lem:monextension}
Let $D\subset[\alpha,\beta]$ be dense and let $\hat U:D\to\mathrm{Sym}(\dd)$ be
$\preceq$-nonincreasing with finite values. Then
$\bar U_t:=\lim_{D\ni s\downarrow t}\hat U_s$ exists for every $t\in[\alpha,\beta)$ and
defines a right-continuous, $\preceq$-nonincreasing extension, locally bounded on
$(\alpha,\beta)$, with at most countably many discontinuities; $\bar U=\hat U$ on $D$
off a countable set, and $\bar U$ is the unique such regularisation.
\end{lem}

\begin{proof}
For $\xi\in\R^\dd$ the function $s\mapsto\langle\hat U_s\xi,\xi\rangle$ is nonincreasing
on $D$; as $s\downarrow t$ it is nondecreasing and bounded above by its value at any
grid point $<t$, hence converges; polarisation defines $\bar U_t$. Monotonicity and
right-continuity are inherited, and local boundedness follows by sandwiching between two
grid values ($A\preceq B\preceq C$ bounds $\|B\|$ by $\max(\|A\|,\|C\|)$). At a
discontinuity the jump $\bar U_{t^-}-\bar U_{t^+}\succeq0$ is nonzero, so its trace, which is a monotone real function of $t$, jumps; hence countability.
\end{proof}

\subsection{Proof of Theorem \ref{thm:aleximpliesdisplacementconvex}}\label{sec:proofforward}

The proof follows the strategy of Petrunin \cite[Proposition 2.2]{petrunin2010alexandrov}, with the essential difference that we never take the trace: all inequalities are kept at the level of quadratic forms on the tangent spaces.

\medskip
\noindent\textbf{Step 0.}
Let $(X,d)$ be an $n$-dimensional Alexandrov space with $\curv\ge0$, and let $\mu_H$ be its $n$-dimensional Hausdorff measure. Let $\mu_0$ and $\mu_1$ be compactly supported probability measures absolutely continuous with respect to the Hausdorff measure.
Let $\theta:X\to\R$ be the Kantorovich potential of \cite[Theorem~1.1]{bertrand2008existence}
for the cost $\tfrac12 d^2$; it is $c$-concave, hence locally semiconcave. Its Hamilton--Jacobi
shifts $\theta_t=H_t\theta$ are given by \eqref{def:theta}. 

By \cite[Theorem~1.1]{bertrand2008existence} the optimal map is $x\mapsto\exp_x(-\nabla\theta(x))$
and the transport rays $\gamma(s)=\exp_x\!\big(-s\,\nabla\theta(x)\big)$ are minimizing geodesics;
moreover $\gamma$ is a $\theta_s$-gradient curve in the sense of
\cite[\S1]{petrunin2010alexandrov}. 
For $\mu_H$-a.e.\ starting point $x$ and a.e.\ parameter
$s$ the point $\gamma(s)$ is regular and lies in $\operatorname{Reg}\theta_s$; at such $s$ we set,
as in Definition~\ref{def:alexentropytensor},
\[
U_s:=\hess_{\gamma(s)}\theta_s ,
\]
a bilinear form on $T_{\gamma(s)}X\cong\R^n$, defined for a.e.\ $s$.
Finally we fix the trivialisation of $T_{\gamma(\cdot)}X$ \emph{along $\gamma$}. We use the construction of Section~\ref{ss:selection}: we fix a compact subinterval $[\alpha,\beta]\subset(0,a)$, a good phase $\varphi$ (Lemma~\ref{lem:goodphase}), and the family
$$
\widehat W_{s\to t}:T_{\gamma(s)}X\to T_{\gamma(t)}X,\quad s<t\in D_\varphi,
$$
of Proposition~\ref{prop:cocycle}: splitting-preserving linear isometries satisfying the
exact cocycle identity \eqref{eq:cocycle} and the normal second-variation bound
$(\mathrm{SV}_\perp)$. All matrix-valued quantities are first formed at grid times: in
a $\widehat W$-parallel orthonormal frame based at the leftmost grid point and adapted
to the splitting (first vector $\dot\gamma/|\dot\gamma|$), each $U_t$ with
$t\in D_\varphi$ yields the tangential scalar
$c_t$ and the normal block $V_t\in\mathrm{Sym}(\dd-1)$ of Definition~\ref{def:alexentropytensor}, assembled into the block-diagonal matrix
$\hat U_t:=\operatorname{diag}(c_t,V_t)\in\mathrm{Sym}(\dd)$; the mixed entries of
$U_t$ are never formed. The passage from grid times to a.e.\ $t$ (needed for
$\ET_t$ to be defined at all) is \emph{not} an additional regularity assumption: it
is supplied by the monotonicity of $\hat U$ on the grid, established in Step~4, together with the canonical extension of
Lemma~\ref{lem:monextension}; we denote by $\bar U$ the resulting right-continuous
(block-diagonal) extension.

Finally, note that with $\ET_t=-\int_0^t\bar U_s\,ds$ as in
Definition~\ref{def:alexentropytensor}, $\dot\ET_t=-\bar U_t$ and
$\ddot\ET_t=-\dot{\bar U}_t$, where $\dot{\bar U}_t=\tfrac{\der}{dt}\bar U_t$ is the
derivative in the fixed $\widehat W$-frame, understood distributionally in $t$. Hence the
matrix displacement convexity $\ddot\ET_t^{\mu_0\to\mu_1}\succeq(\dot\ET_t^{\mu_0\to\mu_1})^2$
is equivalent to
\begin{equation}\label{eq:riccatigoal}
\dot{\bar U}_t\ \preceq\ -\,\bar U_t^2
\end{equation}
in the sense of quadratic forms, understood distributionally along $\gamma$;
$\bar U_t$ being block-diagonal, \eqref{eq:riccatigoal} is exactly the pair of decoupled
inequalities of Definition~\ref{def:alexentropytensor}, and this is what we will prove.

\medskip
\noindent\textbf{Step 1 (Hamilton--Jacobi semigroup).}
For all $0<t_0<t_1$ and all $x\in X$,
\begin{equation}\label{eq:HJsemigroup}
\theta_{t_1}(x) = \underset{y\in X}{\inf}\, \left\{ \theta_{t_0}(y) + \frac{1}{2 ( t_1 - t_0)}\,d(x,y)^2 \right\}.
\end{equation}
Indeed the inequality ``$\le$'' is the triangle inequality applied to the definition
\eqref{def:theta}, and equality holds because $X$ is a length space
\cite[\S2, eq.~(2)]{petrunin2010alexandrov}.

\medskip
\noindent\textbf{Step 2 (the ray realizes the infimum).}
Fix interior parameters $0<t_0<t_1$ and set $q=\gamma(t_0)$, $p=\gamma(t_1)$. Since $\gamma$ is the
$\theta_s$-gradient ray of the optimal transport, the point $q$ is the minimizer in
\eqref{eq:HJsemigroup} at $x=p$; equivalently
\begin{equation}\label{eq:G0}
\theta_{t_1}(p)=\theta_{t_0}(q)+\frac{1}{2(t_1-t_0)}\,d(p,q)^2 .
\end{equation}

\medskip
\noindent\textbf{Step 3 (Hessian comparison, blockwise).}
We prove that, for every grid pair $t_0<t_1$ in $D_\varphi$, so that
$q=\gamma(t_0)\in\operatorname{Reg}\theta_{t_0}$ and
$p=\gamma(t_1)\in\operatorname{Reg}\theta_{t_1}$ automatically, by
Lemma~\ref{lem:goodphase}, it holds
\begin{equation}\label{eq:Hessiancompar}
\hess_{\gamma(t_1)} \theta_{t_1}(u,u) \leq  \hess_{\gamma(t_0)} \theta_{t_0}(v,v) + \frac{\left|u-W(v) \right|^2}{t_1-t_0}
\end{equation}
with $W=\widehat W_{t_0\to t_1}$, for all pairs of vectors that are \emph{either both
normal} ($u\in L_{\gamma(t_1)}$, $v\in L_{\gamma(t_0)}$) \emph{or both tangential}
($u\parallel\dot\gamma(t_1)$, $v\parallel\dot\gamma(t_0)$). 

In both cases the starting point is the same. For $u\in T_pX$, $v\in T_qX$ and small
$\varepsilon\in\R$, set $x_\varepsilon=\exp_p(\varepsilon u)$,
$y_\varepsilon=\exp_q(\varepsilon v)$, where $\exp_p(\varepsilon u)$ denotes the point
reached at time $|\varepsilon|$ along the geodesic issued from $p$ in the direction of
$\operatorname{sgn}(\varepsilon)\,u$, and define
\[
G(\varepsilon) := \theta_{t_1}(x_\varepsilon) - \theta_{t_0}(y_\varepsilon) - \frac{1}{2 ( t_1 - t_0)}\,d(x_\varepsilon,y_\varepsilon)^2.
\]
By \eqref{eq:HJsemigroup}, $G(\varepsilon)\le0$ for every $\varepsilon$, while $G(0)=0$ by
\eqref{eq:G0}. The two potential terms are twice differentiable at $\varepsilon=0$ in the sense of \eqref{eq:hessdefprelim}
(applied at $p$ and at $q$, both regular), so we get
\begin{align}
\theta_{t_1}(x_\varepsilon)&=\theta_{t_1}(p)+\varepsilon\,d_p\theta_{t_1}(u)+\tfrac{\varepsilon^2}2\,\hess_p\theta_{t_1}(u,u)+o(\varepsilon^2),\label{eq:taylorp}\\
\theta_{t_0}(y_\varepsilon)&=\theta_{t_0}(q)+\varepsilon\,d_q\theta_{t_0}(v)+\tfrac{\varepsilon^2}2\,\hess_q\theta_{t_0}(v,v)+o(\varepsilon^2).\label{eq:taylorq}
\end{align}
The squared-distance term, however, is only semiconcave in $\varepsilon$ and need
\emph{not} be twice differentiable at $\varepsilon=0$. The two cases now diverge: in
the tangential case the squared distance is computed \emph{exactly} and $G$ itself
serves as the test function; in the normal case we bound the squared distance from
above by the Euclidean model of Proposition~\ref{prop:cocycle}(iii) and argue by a
barrier, paying attention to the fact that the model bound is available only along the
\emph{positive} scale sequence $\varepsilon_j=1/j$, for $\omega$-almost all $j$.

\smallskip
\emph{Tangential case.} Let $u=\lambda\,\dot\gamma(t_1)/|\dot\gamma|$,
$v=\mu\,\dot\gamma(t_0)/|\dot\gamma|$ with $\lambda,\mu\in\R$. CBB spaces are
non-branching, so the geodesic issued from $p$ in the direction $\pm\dot\gamma(t_1)$ is
the ray itself: $x_\varepsilon=\gamma(t_1+\varepsilon\lambda/|\dot\gamma|)$ and
$y_\varepsilon=\gamma(t_0+\varepsilon\mu/|\dot\gamma|)$, defined two-sidedly for small
$\varepsilon$ since $t_0,t_1$ are interior to $(0,a)$. Both points lie on the
minimizing geodesic $\gamma$, so
\[
d(x_\varepsilon,y_\varepsilon)\ =\ d(p,q)+\varepsilon(\lambda-\mu)
\qquad\text{for all small }\varepsilon .
\]
Hence $G$ is twice differentiable at $\varepsilon=0$ by
\eqref{eq:taylorp}--\eqref{eq:taylorq}, and $\varepsilon=0$ is an interior two-sided
maximum: $G'(0)=0$ and $G''(0)\le0$, the latter reading
\[
\hess_p\theta_{t_1}(u,u)\ \le\ \hess_q\theta_{t_0}(v,v)+\frac{(\lambda-\mu)^2}{t_1-t_0}\,,
\]
which is \eqref{eq:Hessiancompar}: $W$ maps $\dot\gamma(t_0)/|\dot\gamma|$ to
$\dot\gamma(t_1)/|\dot\gamma|$, so $|u-W(v)|=|\lambda-\mu|$. 

\smallskip
\emph{Normal case.} Let $u\in L_{\gamma(t_1)}$, $v\in L_{\gamma(t_0)}$. First, both
linear terms in \eqref{eq:taylorp}--\eqref{eq:taylorq} vanish identically: along the
transport ray $\nabla\theta_{t_1}(p)=-\dot\gamma(t_1)$ and
$\nabla\theta_{t_0}(q)=-\dot\gamma(t_0)$ (Section~\ref{sec:prelim}), the differentials
at the regular points $p,q$ are linear, and $u\perp\dot\gamma(t_1)$,
$v\perp\dot\gamma(t_0)$ give
\begin{equation}\label{eq:linvanish}
d_p\theta_{t_1}(u)=\langle-\dot\gamma(t_1),u\rangle=0,
\qquad
d_q\theta_{t_0}(v)=\langle-\dot\gamma(t_0),v\rangle=0 .
\end{equation}
Next, Proposition~\ref{prop:cocycle}(iii) provides reference points
$\tilde p,\tilde q\in\R^\dd$ with $|\tilde p-\tilde q|=d(p,q)$ and linear isometries
$\iota_p:T_pX\to\R^\dd$, $\iota_q:T_qX\to\R^\dd$ inducing $W$ (i.e.\
$\iota_p\circ W=\iota_q$, the intertwining relation read in linear parts) and taking
the ray directions to the direction of $\tilde q-\tilde p$, such that, with
$\varepsilon=\varepsilon_j$,
\[
d(x_\varepsilon,y_\varepsilon)\ \le\ \big|(\tilde p+\varepsilon\,\iota_p u)-(\tilde q+\varepsilon\,\iota_q v)\big|+o(\varepsilon^2)
\qquad\text{for $\omega$-almost all }j .
\]
Squaring (both sides equal $d(p,q)+O(\varepsilon)$, so the error remains $o(\varepsilon^2)$) and
using $\iota_p u-\iota_q v=\iota_p(u-W v)$ together with the fact that $\iota_p$ is an isometry,
\begin{equation}\label{eq:secondvarmodel}
d(x_\varepsilon,y_\varepsilon)^2\ \le\ d(p,q)^2+2\varepsilon\,\big\langle\tilde p-\tilde q,\ \iota_p(u-W v)\big\rangle+\varepsilon^2\,|u-W v|^2+o(\varepsilon^2),
\end{equation}
where the linear term \emph{also} vanishes identically: $u-Wv$ is normal ($W$
preserves the splitting), so $\iota_p(u-Wv)\perp\tilde q-\tilde p$, i.e.
\begin{equation}\label{eq:modellinvanish}
\big\langle\tilde p-\tilde q,\ \iota_p(u-W v)\big\rangle=0 .
\end{equation}
Now introduce the \emph{barrier}
\[
\widetilde G(\varepsilon):=\theta_{t_1}(x_\varepsilon)-\theta_{t_0}(y_\varepsilon)
-\frac{1}{2(t_1-t_0)}\Big[d(p,q)^2+\varepsilon^2|u-W v|^2\Big].
\]
Because the coefficient $\tfrac1{2(t_1-t_0)}$ is positive and
\eqref{eq:secondvarmodel}--\eqref{eq:modellinvanish} bound the bracket from below by
$d(x_{\varepsilon_j},y_{\varepsilon_j})^2-o(\varepsilon_j^2)$, we have
\[
\widetilde G(\varepsilon_j)\ \le\ G(\varepsilon_j)+o(\varepsilon_j^2)\ \le\ o(\varepsilon_j^2)
\qquad\text{for $\omega$-almost all }j ,
\]
while $\widetilde G(0)=G(0)=0$. On the other hand, by
\eqref{eq:taylorp}--\eqref{eq:taylorq} and \eqref{eq:linvanish} the barrier admits the
second-order expansion \emph{with no linear term},
\[
\widetilde G(\varepsilon)=\tfrac12\,\widetilde G''(0)\,\varepsilon^2+o(\varepsilon^2),
\]
\[
\widetilde G''(0)=\hess_p\theta_{t_1}(u,u)-\hess_q\theta_{t_0}(v,v)-\frac{|u-W v|^2}{t_1-t_0}.
\]
Combining the two displays: for every $\delta>0$ the set of indices $j$ with
$\tfrac12\widetilde G''(0)\,\varepsilon_j^2\le\delta\,\varepsilon_j^2$ belongs to
$\omega$, in particular is nonempty, whence $\widetilde G''(0)\le2\delta$; letting
$\delta\downarrow0$ gives $\widetilde G''(0)\le0$, which is \eqref{eq:Hessiancompar}.

\medskip
\noindent\textbf{Step 4 (monotonicity, regularisation, and conclusion).}
Read all the retained Hessian
data, $t\in D_\varphi$, in the common $\widehat W$-parallel splitting-adapted
orthonormal frame based at the leftmost grid point (Step~0); they become the
block-diagonal matrices $\hat U_t=\operatorname{diag}(c_t,V_t)$, and for any grid pair
$t_0<t_1$ the transport $W=\widehat W_{t_0\to t_1}$ is the identity of $\R^n$ (it
preserves the splitting, so the frame remains adapted along the ray). By
\eqref{eq:cocycle} this is consistent across all grid pairs simultaneously. 

Reading \eqref{eq:Hessiancompar} in this frame (so $W=\Id$) and writing
$\Delta:=t_1-t_0>0$, it states
\begin{equation}\label{eq:Hessallframe}
\langle \hat U_{t_1}u,u\rangle\ \le\ \langle \hat U_{t_0}v,v\rangle+\frac1\Delta|u-v|^2
\qquad(\forall\,u,v\in E),
\end{equation}
for every grid pair $t_0<t_1$ in $D_\varphi$, where $E$ denotes either the tangential
line $\R e_1$ or the normal subspace $e_1^\perp\cong\R^{n-1}$. We draw two consequences, blockwise. First, for fixed $u\in E$, since its quadratic part in $v$ is $\langle(\hat U_{t_0}+\tfrac1\Delta\Id)v,v\rangle|_E$, this forces
$\Id+\Delta\hat U_{t_0}\succeq0$ on each of the two blocks, hence as a matrix, $\hat
U_{t_0}$ being block-diagonal. Applying this with a strictly larger grid time
$t_1'\in(t_1,\beta)\cap D_\varphi$ gives $\Id+(t_1'-t_0)\hat U_{t_0}\succeq0$, whence
\[
\Id+\Delta\,\hat U_{t_0}\ \succeq\ \Bigl(1-\tfrac{\Delta}{\,t_1'-t_0\,}\Bigr)\Id\ \succ\ 0
\]
for every grid pair with $t_1<\beta$. Second, minimising the right-hand side of
\eqref{eq:Hessallframe} over $v\in E$ is an inf-convolution of quadratic forms: by the following elementary identity, valid on a Euclidean space $E$ for a symmetric matrix $B$ with $\Id_E+\Delta B\succ0$,
\[
\inf_{v\in E}\Bigl[\langle Bv,v\rangle+\tfrac1\Delta|u-v|^2\Bigr]
=\bigl\langle B(\Id_E+\Delta B)^{-1}u,\,u\bigr\rangle,
\]
applied once with $B=c_{t_0}$ on $E=\R e_1$ and once with $B=V_{t_0}$ on
$E=e_1^\perp$, we obtain the two blockwise bounds
$c_{t_1}\le c_{t_0}(1+\Delta c_{t_0})^{-1}$ and
$V_{t_1}\preceq V_{t_0}(\Id+\Delta V_{t_0})^{-1}$, which assemble (the map
$B\mapsto B(\Id+\Delta B)^{-1}$ acting block-by-block on block-diagonal matrices) into
the \emph{operator} bound
\begin{equation}\label{eq:matrixriccatiforward}
\hat U_{t_1}\ \preceq\ \hat U_{t_0}\bigl(\Id+(t_1-t_0)\hat U_{t_0}\bigr)^{-1}
\ \preceq\ \hat U_{t_0},
\qquad t_0<t_1<\beta\ \text{ in }D_\varphi ,
\end{equation}
the second inequality because
$B(\Id+\Delta B)^{-1}=B-\Delta B^2(\Id+\Delta B)^{-1}\preceq B$ when
$\Id+\Delta B\succ0$. In particular $t\mapsto\hat U_t$ is
$\preceq$-\emph{nonincreasing on the grid}, and Lemma~\ref{lem:monextension} yields its
right-continuous, $\preceq$-nonincreasing, locally bounded extension $\bar U$ to
$[\alpha,\beta)$, with at most countably many discontinuities --- the extension mentioned
in Step~0, which renders $\ET_t=-\int\bar U$ well defined. Approximating any two
continuity points $t_0<t_1$ of $\bar U$ from the right by grid points, and using the
continuity of $B\mapsto B(\Id+\Delta B)^{-1}$ on $\{\Id+\Delta B\succ0\}$ (locally
uniformly in $\Delta$), the comparison \eqref{eq:matrixriccatiforward} extends off the grid: for all but countably many $t_0<t_1$\ in $(\alpha,\beta)$,
\begin{equation}\label{eq:matrixriccatibar}
\bar U_{t_1}\ \preceq\ \bar U_{t_0}\bigl(\Id+(t_1-t_0)\bar U_{t_0}\bigr)^{-1}.
\end{equation}

It remains to pass from the finite comparison \eqref{eq:matrixriccatibar} to the
differential inequality \eqref{eq:riccatigoal}. For a fixed vector $\xi\in\R^n$,
expanding the right-hand side of \eqref{eq:matrixriccatibar} to first order in $\Delta$
\[
\bar U_{t_0}\bigl(\Id+\Delta \bar U_{t_0}\bigr)^{-1}=\bar U_{t_0}-\Delta\,\bar U_{t_0}^2+O(\Delta^2),
\]
gives $\langle \bar U_{t_1}\xi,\xi\rangle-\langle \bar U_{t_0}\xi,\xi\rangle\le-\Delta\,|\bar U_{t_0}\xi|^2
+O(\Delta^2)$. Dividing by $\Delta$ and integrating against an arbitrary $0\le\chi\in
C_c^\infty(\alpha,\beta)$ yields
\[
\int_\alpha^{\beta}\Bigl(\langle \bar U_s\xi,\xi\rangle\,\chi'(s)-|\bar U_s\xi|^2\,\chi(s)\Bigr)\,ds\ \ge\ 0 ,
\]
i.e.\ $\tfrac{d}{ds}\langle \bar U_s\xi,\xi\rangle\le-|\bar U_s\xi|^2=-\langle \bar U_s^2\xi,\xi\rangle$
distributionally on $(\alpha,\beta)$. As $\xi\in\R^n$ was arbitrary, and as
$[\alpha,\beta]\subset(0,a)$ was an arbitrary compact subinterval, this is exactly
\eqref{eq:riccatigoal}:
\[
\dot{\bar U}_t\ \preceq\ -\,\bar U_t^2
\]
in the Loewner order, distributionally along $\gamma$, which proves the theorem.
\qed

\begin{rem}[consistency of the regularisation]\label{rem:consistency}
Note that the extension $\bar U=\operatorname{diag}(\bar c,\bar V)$ agrees with the intrinsic
blockwise Hessian data wherever the latter are defined, in the following
gauge-invariant sense: for a.e.\ $t$,
$\bar c_t=\hess_{\gamma(t)}\theta_t(\dot\gamma,\dot\gamma)/|\dot\gamma|^2$ and the
ordered eigenvalues of $\bar V_t$ coincide with those of the normal block
$\hess_{\gamma(t)}\theta_t|_{L_{\gamma(t)}\times L_{\gamma(t)}}$ with respect to
$\met_{\gamma(t)}$. 
Note also that since an a.e.\ nonincreasing function agrees a.e.\ with its right-continuous
monotone representative, which coincides with the corresponding branch of $\bar U_t$
on the dense grid, hence a.e. In particular
$\Tr\bar U_t=\bar c_t+\Tr\bar V_t=\Tr\hess_{\gamma(t)}\theta_t$ for a.e.\ $t$ (the
discarded mixed entries do not enter the trace), so the trace of the present
construction recovers the scalar theory of \cite{petrunin2010alexandrov}.
\end{rem}

\section{Matrix displacement convexity implies CBB(0)}\label{sec:converse}

The goal of this section is to prove the converse direction of \eqref{eq:ARS} in the Alexandrov synthetic setting.

\begin{thm}\label{thm:displacementconveximpliesAlex}
Let $(X,d)$ be an Alexandrov space with curvature bounded below by some
$\kappa > - \infty$. Suppose that for all compactly supported probability measures
$\mu_0$ and $\mu_1$ which are absolutely continuous with respect to the Hausdorff
measure \emph{there exists}, along $\mu_H$-a.e.\ transport ray, an admissible parallel
trivialisation (Definition~\ref{def:alexentropytensor}) in which the block-diagonal entropy tensor
$\ET_t^{\mu_0\to\mu_1}$ is matrix displacement convex. Then $(X,d)$ is an
Alexandrov space with non-negative curvature (\textit{i.e.}\ $\kappa$ can be taken to
be zero).
\end{thm}

First, let us show the existence of admissible trivialisations on $\mathrm{CBB}(\kappa)$ spaces, $\kappa>-\infty$. Note that such trivialisations are not required to satisfy the second variation formula, contrary to the forward proof of Section \ref{sec:forward}.

\begin{lem}[existence of admissible trivialisations]\label{lem:existadmissible}
Let $(X,d)$ be a finite-dimensional $\mathrm{CBB}(\kappa)$ space, $\kappa>-\infty$, and
let $\mu_0,\mu_1$ be compactly supported and absolutely continuous. Then along
$\mu_H$-a.e.\ transport ray $\gamma$ there exists an admissible parallel trivialisation $\tau$ in the sense of Definition~\ref{def:alexentropytensor}. Moreover, for every such
trivialisation, the matrix $U_s$ of $\hess_{\gamma(s)}\theta_s$ in a $\tau$-parallel
orthonormal frame satisfies the \emph{two-sided interior bound}
\begin{equation}\label{eq:twosidedhess}
-\frac{\lambda_0}{s_1-s}\,\Id\ \preceq\ U_s\ \preceq\ \frac{\lambda_0}{s-s_0}\,\Id
\qquad\text{for a.e.\ }s\in(s_0,s_1),
\end{equation}
where $(s_0,s_1)$ is any parameter interval whose endpoints are interior to the ray and
$\lambda_0=\lambda_0(\kappa,\mathrm{diam})$ is the semiconcavity constant of
$\tfrac12d^2$ on the (compact) transport set.
\end{lem}

\begin{proof}
\emph{Existence of a cocycle frame.} By Lemma~\ref{lem:nullrays} of
Appendix~\ref{app:rays} (with $N_t:=S_X\cup(X\setminus\operatorname{Reg}\theta_t)$, as
in Lemma~\ref{lem:goodphase}), for $\mu_H$-a.e.\ ray the set
$D_\gamma:=\{s:\gamma(s)\ \text{is a good time}\}$ has full measure, in particular is
dense. Fix $s_*\in D_\gamma$. On the regular set, the Otsu--Shioya $C^0$-Riemannian
structure \cite{OtsuShioya94} provides charts in which the metric tensor is a
measurable (indeed continuous on the chart domain) field of positive matrices; applying
Gram--Schmidt at $\gamma(s)$ to the family consisting of $\dot\gamma(s)/|\dot\gamma|$
followed by the coordinate frame --- so that the first vector of the resulting basis
is $\dot\gamma(s)/|\dot\gamma|$ and the remaining ones span $L_{\gamma(s)}$ --- yields
a family of splitting-adapted orthonormal
bases of $T_{\gamma(s)}X$, $s\in D_\gamma$, measurable in $s$, i.e.\ a measurable
family of linear isometries $I_s:\R^\dd\to T_{\gamma(s)}X$ with
$I_s(e_1)=\dot\gamma(s)/|\dot\gamma|$. Set
\[
\tau_{s\to t}:=I_t\circ I_s^{-1},\qquad s,t\in D_\gamma .
\]
Being a coboundary of adapted frames, this family satisfies the exact cocycle identity
and the splitting-preservation. Note that this need not be related to parallelism in any
differential-geometric sense: admissibility, Definition~\ref{def:alexentropytensor},
only requires the cocycle, the splitting-preservation and the boundedness below.

\emph{Two-sided bound.} Local boundedness of the matrices of $U_s$ is frame-independent
($\|U_s\|$ is an intrinsic quantity), so it suffices to prove
\eqref{eq:twosidedhess}; this also proves the ``for every''. Fix $s\in D_\gamma\cap(s_0,s_1)$ and
write $p:=\gamma(s)$. \emph{Upper bound:} by \eqref{eq:HJsemigroup},
$\theta_s=\inf_y\{\theta_{s_0}(y)+\tfrac1{2(s-s_0)}d(\cdot,y)^2\}$ is an infimum of
$\tfrac{\lambda_0}{s-s_0}$-concave functions, hence is itself
$\tfrac{\lambda_0}{s-s_0}$-concave on the transport set; at the point $p$, where the
second-order expansion \eqref{eq:hessdefprelim} holds, this gives
$U_s\preceq\tfrac{\lambda_0}{s-s_0}\Id$. \emph{Lower bound:} by \eqref{eq:HJsemigroup}
applied between $s$ and $s_1$, for every $z$,
\[
\theta_s(x)\ \ge\ \theta_{s_1}(z)-\frac{d(x,z)^2}{2(s_1-s)}=:\psi_z(x),
\]
and by the no-loss identity \eqref{eq:G0} equality holds at $x=p$ for
$z^*:=\gamma(s_1)$. The function $\psi:=\psi_{z^*}$ touches $\theta_s$ from below at
$p$. Since $p$ is interior to the geodesic from $\gamma(s_0)$ to $z^*$ and geodesics in
Alexandrov spaces do not branch, the shortest path $p\,z^*$ is unique, so $\psi$ is
differentiable at $p$; $\theta_s$ is differentiable at $p$ as well ($p\in
\operatorname{Reg}\theta_s$), and at the touching point the differentials agree on the
(dense) set of geodesic directions, hence everywhere on $T_pX$ by linearity. For a
direction $v$ admitting a geodesic $\sigma(\varepsilon)=\exp_p(\varepsilon v)$,
$\varepsilon\in[0,\varepsilon_0]$, the $\tfrac{\lambda_0}{s_1-s}$-concavity of $-\psi$
along $\sigma$ gives the one-sided bound
$\psi(\sigma(\varepsilon))\ge\psi(p)+\varepsilon\,d_p\psi(v)
-\tfrac{\lambda_0}{2(s_1-s)}\varepsilon^2|v|^2$, whence, comparing with the second-order expansion of $\theta_s$ at $p$ and using $\theta_s\ge\psi$, $\theta_s(p)=\psi(p)$,
$d_p\theta_s=d_p\psi$:
\[
\tfrac12\<U_sv,v\>\ \ge\ -\tfrac{\lambda_0}{2(s_1-s)}|v|^2 .
\]
Geodesic directions being dense in $\Sigma_p$ and both sides continuous in $v$, the
bound holds for all $v\in T_pX$. 

\emph{Admissibility.} The matrices $\hat U_s:=I_s^{-1}U_sI_s$, $s\in D_\gamma$, are
measurable in $s$ (composition of the measurable frame with the measurable Hessian
field; joint measurability of the latter is part of Lemma~\ref{lem:nullrays}, see
Appendix~\ref{app:rays}) and locally bounded by \eqref{eq:twosidedhess}; so are, a
fortiori, their tangential scalars $c_s$ and normal blocks $V_s$. Any
measurable locally bounded extension to a.e.\ $s$ (e.g.\ by $0$ off $D_\gamma$, which
is already of full measure) satisfies Definition~\ref{def:alexentropytensor}.
\end{proof}


We now turn to the proof of Theorem \ref{thm:displacementconveximpliesAlex}.
The proof proceeds in two main steps. First, in Section \ref{ss:MDCtoC}, we show that the matrix displacement convexity suffices to obtain the Hessian upper bound \eqref{eq:condC} for the squared distance function. Second, in section \ref{ss:CtoAlex}, we show that this latter Hessian bound implies a non-negative lower bound on the sectional curvature in the sense of Alexandrov.

\subsection{Matrix displacement convexity implies squared-distance bound}\label{ss:MDCtoC}

In this section, we will control, for every point $p$, the squared distance to $p$ along every geodesic, i.e.\ the
one-sided Hessian bound
\begin{equation}\label{eq:condC}
\hess_q\big(\tfrac12 d^2(\cdot,p)\big)\ \preceq\ \Id
\qquad(\text{in the support sense, for every }p,q).
\end{equation}
We obtain \eqref{eq:condC}
directly from a matrix Riccati comparison applied to a transport that \emph{focuses at the point} $p$ and read \emph{backwards} from the focal end.

Throughout this subsection we read the entropy-tensor integrand along a transport ray in the
form $U_t=\dot{\mathcal A}_t\,\mathcal A_t^{-1}$, where $\mathcal A_t$ is the matrix transport
Jacobi field with $\mathcal A_0=\Id$. Equivalently $U_t$ is the Hessian of the McCann velocity
potential, the normalisation for which the Jacobi identity \eqref{eq:traceentropy} holds.  Matrix displacement convexity reads $\Rs_t:=-\dot U_t-U_t^2\succeq0$. With this normalisation,
for the transport \emph{to} the point $p$ (Kantorovich potential $\tfrac12 d^2(\cdot,p)$) one
has $U_0=-\hess_x\tfrac12 d^2(\cdot,p)$, and the ray focuses at $p$.

Two readings of the hypothesis must be kept apart, the entropy tensor of
Definition~\ref{def:alexentropytensor} being block-diagonal. The \emph{block reading},
which is what Theorem~\ref{thm:displacementconveximpliesAlex} assumes, is
$\hat\Rs_t:=-\dot{\hat U}_t-\hat U_t^2\succeq0$ for the block matrix
$\hat U_t=\operatorname{diag}(c_t,V_t)$ read in the given admissible trivialisation; it
decouples into a scalar and a normal-block Riccati inequality, and all the matrix
lemmas below (Lemmas~\ref{lem:matrixriccati} and~\ref{lem:backwardriccati}) apply to
each block separately, being pure matrix statements in arbitrary dimension. The
\emph{full reading}, $\Rs_t=-\dot U_t-U_t^2\succeq0$ for the full matrix $U_t$
including its mixed entries (the hypothesis of Remark~\ref{rem:fullSV}), is strictly
stronger. Both readings yield \eqref{eq:condC}: the full reading directly
(Steps~1--2), the block reading after the focal control of the source mixed entries
(Lemma~\ref{lem:focalmixed}) supplies the missing off-diagonal information.

We first state a matrix Riccati comparison theorem valid
for a \emph{matrix-valued Radon-measure} coefficient and a merely $BV$ Riccati solution. For the sake of completeness, we give a self-contained proof by the index form.

\begin{lem}[Matrix Riccati comparison, measure coefficient]\label{lem:matrixriccati}
Let $Q$ be a nonnegative symmetric matrix-valued Radon measure on $[t_0,t_1]$ and let $U$ be a
symmetric, bounded $BV$ function on $[t_0,t_1]$ solving, distributionally,
\[
\dot U_t\ =\ -\,U_t^2-Q_t .
\]
Let $W_t:=\dot{\mathcal B}_t\,\mathcal B_t^{-1}$ with $\mathcal B_t:=\Id+(t-t_0)\,U_{t_0}$ be the
flat ($Q\equiv0$) solution with $W_{t_0}=U_{t_0}$, and assume $\mathcal B_t$ is invertible (i.e.\
$W$ is finite) on all of $[t_0,t_1]$. Then
\[
U_{t_1}\ \preceq\ W_{t_1}\ =\ U_{t_0}\bigl(\Id+(t_1-t_0)U_{t_0}\bigr)^{-1}.
\]
If moreover $U_{t_0}\succ0$, then $W_{t_1}=\bigl(U_{t_0}^{-1}+(t_1-t_0)\Id\bigr)^{-1}$.
\end{lem}

\begin{proof}
Let $\mathcal A$ be the absolutely continuous solution of $\dot{\mathcal A}_t=U_t\mathcal A_t$,
$\mathcal A_{t_0}=\Id$ 
Since $U$ is bounded, $\dot{\mathcal A}$ is
$BV$, and as measures
\[
\ddot{\mathcal A}=\dot U\,\mathcal A+U\dot{\mathcal A}=(-U^2-Q)\mathcal A+U^2\mathcal A=-Q\,\mathcal A .
\]
Likewise $\ddot{\mathcal B}=0$ with $\mathcal B_{t_0}=\Id$, $\dot{\mathcal B}_{t_0}=U_{t_0}$, so
$\mathcal A,\mathcal B$ share the same initial data $(\Id,U_{t_0})$.

Fix $u\in\R^n$ and let $J(t):=\mathcal A_t\,\mathcal A_{t_1}^{-1}u$; then $J$ is AC with $\dot J$ of
bounded variation, $\ddot J=-Q\,J$ as a measure, $J(t_1)=u$, and at $t_0$
\[
\dot J(t_0)=\dot{\mathcal A}_{t_0}\mathcal A_{t_1}^{-1}u=U_{t_0}\,\mathcal A_{t_1}^{-1}u
=U_{t_0}\,J(t_0).
\]
The Stieltjes integration-by-parts formula for the $BV$ pairing $t\mapsto\langle\dot J,J\rangle$
gives the \emph{index-form identity}
\begin{align*}
\langle U_{t_1}u,u\rangle & =\langle\dot J(t_1),J(t_1)\rangle\\
&= \langle U_{t_0}J(t_0),J(t_0)\rangle+\int_{t_0}^{t_1}|\dot J|^2\,dt-\int_{[t_0,t_1]}\langle dQ\,J,J\rangle\\
&=:I(J),
\end{align*}
where we used $d\langle\dot J,J\rangle=|\dot J|^2\,dt+\langle d\dot J,J\rangle$ and
$d\dot J=\ddot J=-dQ\,J$. The same computation with $Q\equiv0$ for $\bar J(t):=\mathcal
B_t\mathcal B_{t_1}^{-1}u$ (which satisfies $\ddot{\bar J}=0$, $\dot{\bar
J}(t_0)=U_{t_0}\bar J(t_0)$, $\bar J(t_1)=u$) gives
\[
\langle W_{t_1}u,u\rangle=\langle U_{t_0}\bar J(t_0),\bar J(t_0)\rangle
+\int_{t_0}^{t_1}|\dot{\bar J}|^2\,dt=I_0(\bar J),
\]
$I_0$ being the index form with $Q$ omitted.

Write any AC field $\eta$ with $\eta(t_1)=u$ as $\eta=J+\zeta$, $\zeta(t_1)=0$. Integrating by
parts and using $\ddot J=-dQ\,J$ together with the natural boundary condition $\dot
J(t_0)=U_{t_0}J(t_0)$, the cross term vanishes:
\[
I(\eta)=I(J)+\Bigl[\langle U_{t_0}\zeta(t_0),\zeta(t_0)\rangle+\int_{t_0}^{t_1}|\dot\zeta|^2\,dt
-\int_{[t_0,t_1]}\langle dQ\,\zeta,\zeta\rangle\Bigr].
\]
The bracket is nonnegative for \emph{every} $\zeta$ with $\zeta(t_1)=0$: using
$d\langle U\zeta,\zeta\rangle=-|U\zeta|^2\,dt-\langle dQ\,\zeta,\zeta\rangle+2\langle
U\zeta,\dot\zeta\rangle\,dt$ (which is just $\dot U=-U^2-Q$ paired with $\zeta$), one gets
$|\dot\zeta|^2\,dt-d\langle U\zeta,\zeta\rangle=|\dot\zeta-U\zeta|^2\,dt+\langle dQ\,\zeta,\zeta
\rangle$, and integrating over $[t_0,t_1]$ with $\zeta(t_1)=0$, $U_{t_0}$ as the boundary value of
$U$,
\begin{equation}\label{eq:indexpositive}
\langle U_{t_0}\zeta(t_0),\zeta(t_0)\rangle+\int_{t_0}^{t_1}|\dot\zeta|^2\,dt
-\int_{[t_0,t_1]}\langle dQ\,\zeta,\zeta\rangle
=\int_{t_0}^{t_1}|\dot\zeta-U\zeta|^2\,dt\ \ge\ 0 ,
\end{equation}
where $U$ is bounded on $[t_0,t_1]$ (so the right-hand side is finite) because $\mathcal A$ is
invertible there. Thus the bracket is $\ge0$ and $J$ minimises the index form: $I(J)\le I(\eta)$
for every $\eta$ with $\eta(t_1)=u$. In particular $I(J)\le I(\bar J)$.

Finally, since $Q\succeq0$,
\begin{align*}
I(\bar J) &=\langle U_{t_0}\bar J(t_0),\bar J(t_0)\rangle+\int_{t_0}^{t_1}|\dot{\bar J}|^2\,dt
-\int_{[t_0,t_1]}\langle dQ\,\bar J,\bar J\rangle\\
& \le\ \langle U_{t_0}\bar J(t_0),\bar J(t_0)\rangle+\int_{t_0}^{t_1}|\dot{\bar J}|^2\,dt\\
& =\langle W_{t_1}u,u\rangle ,
\end{align*}
the last equality being the index-form identity for the flat field $\bar J$ (apply the first
display with $Q\equiv0$). Combining, $\langle U_{t_1}u,u\rangle=I(J)\le I(\bar J)\le\langle
W_{t_1}u,u\rangle$. As $u$ was arbitrary, $U_{t_1}\preceq W_{t_1}$. Finally
$W_{t_1}=\dot{\mathcal B}_{t_1}\mathcal B_{t_1}^{-1}=U_{t_0}(\Id+(t_1-t_0)U_{t_0})^{-1}$, which
equals $(U_{t_0}^{-1}+(t_1-t_0)\Id)^{-1}$ when $U_{t_0}\succ0$.
\end{proof}

\begin{lem}[Backward matrix Riccati comparison]\label{lem:backwardriccati}
Let $U_t$ be a symmetric, bounded $BV$ solution of $\dot U_t=-U_t^2-\Rs_t$ on $[t_0,t_1]$,
where $\Rs_t\succeq0$ is a nonnegative symmetric matrix-valued Radon measure. If $U_{t_1}\prec0$, then
$U_{t_1}^{-1}-(t_1-t_0)\Id\prec0$ and
\begin{equation}\label{eq:backwardriccati}
U_{t_0}\ \succeq\ \big(U_{t_1}^{-1}-(t_1-t_0)\Id\big)^{-1}.
\end{equation}
\end{lem}

\begin{proof}
We reduce to the forward comparison of Lemma~\ref{lem:matrixriccati} by time reversal, which
turns the focal (backward) problem into a regular initial-value one and avoids any
smallest-eigenvalue barrier. For $s\in[0,t_1-t_0]$ set
\[
\tilde U_s:=-\,U_{t_1-s},\qquad \tilde\Rs_s:=\Rs_{t_1-s} .
\]
Then $\tilde U$ is symmetric, bounded and $BV$, $\tilde\Rs\succeq0$ is a nonnegative matrix Radon
measure, and, since reversing time sends $\dot U_t=-U_t^2-\Rs_t$ to
\[
\tfrac{d}{ds}\tilde U_s=\dot U_{t_1-s}=-U_{t_1-s}^2-\Rs_{t_1-s}=-\tilde U_s^2-\tilde\Rs_s ,
\]
$\tilde U$ solves the \emph{same} forward Riccati equation on $[0,t_1-t_0]$, with initial datum
$\tilde U_0=-U_{t_1}\succ0$ (here we use $U_{t_1}\prec0$). The flat comparison solution with this
datum is
\[
\tilde W_s=\tilde U_0\bigl(\Id+s\,\tilde U_0\bigr)^{-1}=\bigl(\tilde U_0^{-1}+s\Id\bigr)^{-1},
\]
which, because $\tilde U_0\succ0$, is positive definite and \emph{finite for every} $s\ge0$ (a
positive-definite flat Riccati solution never blows up forward in time); in particular the
finiteness hypothesis of Lemma~\ref{lem:matrixriccati} holds on all of $[0,t_1-t_0]$.
Lemma~\ref{lem:matrixriccati} gives $\tilde U_{t_1-t_0}\preceq\tilde W_{t_1-t_0}$, i.e.
\[
-\,U_{t_0}\ \preceq\ \bigl((-U_{t_1})^{-1}+(t_1-t_0)\Id\bigr)^{-1}.
\]
Negating 
 yields exactly \eqref{eq:backwardriccati}.
\end{proof}

The second ingredient is the focal degeneration of a transport into a shrinking ball.

\begin{lem}[Focal blow-up of transport to a shrinking ball]\label{lem:focalblowup}
Let $p$ be a regular point, let $\mu_0\in\PR$ be a fixed absolutely continuous measure, and for
$\varepsilon>0$ set $\mu_1^\varepsilon=\mu_H(B(p,\varepsilon))^{-1}\,\mu_H\!\restriction\!
B(p,\varepsilon)$. Let $U^\varepsilon_t$ be the entropy-tensor integrand of the optimal
transport $\mu_0\to\mu_1^\varepsilon$, read in McCann time $t\in[0,1]$, so that
$\mathcal A^\varepsilon_1=DT^\varepsilon$ is the differential of the optimal map and
$U^\varepsilon_1\prec0$. Then for $\mu_0$-a.e.\ $x$,
\[
\big(U^\varepsilon_1\big)^{-1}\ \longrightarrow\ 0\qquad(\varepsilon\to0),
\qquad\text{hence}\qquad
\big(U^{\varepsilon\,-1}_1-\Id\big)^{-1}\ \longrightarrow\ -\Id .
\]
\end{lem}

\begin{proof}
We assemble three standard inputs, cited rather than reproved: (i) existence and uniqueness of the
optimal map on the $\mathrm{CBB}$ space $X$ for the cost $\tfrac12 d^2$, with absolutely
continuous source \cite[Thm.~1.1]{bertrand2008existence}; (ii) stability of optimal plans under
weak convergence of the marginals — as $\mu_1^\varepsilon\rightharpoonup\delta_p$ (and after the
$\varepsilon^{-1}$ rescaling below, $\Phi^\varepsilon_\#\mu_1^\varepsilon\rightharpoonup
\mathrm{unif}(B(0,1))$), any weak limit of the optimal plans is optimal for the limit problem
\cite[Thm.~5.20]{Villani09}, and by the uniqueness (i) the optimal maps themselves converge in
$\mu_0$-measure; (iii) Alexandrov's a.e.\ second differentiability of semiconcave/convex
functions and the a.e.\ differentiability and invertibility of $\exp_p$ at a regular point.

\emph{Isotropic compression.} Rescale the target at $p$ by $\varepsilon^{-1}$ via
\[
\Phi^\varepsilon=\varepsilon^{-1}\exp_p^{-1}\colon B(p,\varepsilon)\to B(0,1)\subset T_pX.
\]
The maps $\tilde T^\varepsilon:=\Phi^\varepsilon\circ T^\varepsilon$ transport $\mu_0$ to
$\mathrm{unif}(B(0,1))$, and by the first-order expansion
\[
\tfrac12 d\big(x,\exp_p(\varepsilon w)\big)^2=\tfrac12 d(x,p)^2-\varepsilon\,\langle
\exp_p^{-1}(x),\,w\rangle+O(\varepsilon^2)
\]
they converge to the optimal map $\tilde T^0$ for the limiting cost
$-\langle\exp_p^{-1}(x),w\rangle$ from $\mu_0$ to $\mathrm{unif}(B(0,1))$; writing
$V=\exp_p^{-1}$, one has $\tilde T^0=\nabla g\circ V$ with $g$ convex (Brenier). By Alexandrov's theorem $\hess g$ exists, is finite, and (since $\det\hess g=\rho_{\mathrm{unif}}/\rho_{V_\#\mu_0}>0$
a.e.) is invertible $\mu_0$-a.e.; at the regular point $p$, $DV=D\exp_p^{-1}$ is finite and
invertible a.e. Hence $D\tilde T^0$ is finite and invertible a.e., and
\[
\mathcal A^\varepsilon_1(x)=DT^\varepsilon(x)=\varepsilon\,(D\exp_p)\,D\tilde T^\varepsilon(x)\
\longrightarrow\ 0\qquad(\mu_0\text{-a.e.}),
\]
the factor $\varepsilon$ being scalar, so the convergence to $0$ is isotropic.

\emph{Lower bound at the focal end.} The Hessian at the focal end equals, up to sign, the source
Hessian of the \emph{reversed} transport. Indeed, if $\check\mu_s=\mu_{1-s}$ is the reversed
McCann geodesic, which is an absolutely continuous transport $\mu_1^\varepsilon\to\mu_0$, then
$\nabla\check\phi_s=-\nabla\phi_{1-s}$ at the common point, whence $\check U_0=-U^\varepsilon_1$.
The reversed transport \emph{expands} from the round ball $B(p,\varepsilon)$, and the same
rescaling applied to it gives, by Alexandrov for the convex potential $\check g$ of the reversed
limit map,
\[
\check U_0\ =\ \tfrac1\varepsilon\big(\hess\check g\cdot D\exp_p^{-1}\big)+o(\tfrac1\varepsilon)\
\succeq\ \tfrac{c(x)}{\varepsilon}\,\Id\qquad(\mu_0\text{-a.e.}),\qquad c(x)>0 .
\]
Therefore $-U^\varepsilon_1\succeq\tfrac{c(x)}{\varepsilon}\Id$, so
$\big(U^\varepsilon_1\big)^{-1}\preceq\tfrac{\varepsilon}{c(x)}\Id\to0$, and consequently
$\big(U^{\varepsilon\,-1}_1-\Id\big)^{-1}\to(-\Id)^{-1}=-\Id$.
\end{proof}

Next, we show that the same focal degeneration controls the mixed entries of the source Hessian. The proof uses a geometric identity rather than matrix algebra. 

\begin{lem}[the source mixed entries vanish at the focal rate]\label{lem:focalmixed}
In the setting of Lemma~\ref{lem:focalblowup}, write $u^\varepsilon(x):=\uparrow_x^{T^\varepsilon(x)}$
for the ray direction at the source. Then at $\mu_0$-a.e.\ $x$, for every
$w\in(u^\varepsilon)^{\perp}$,
\[
\bigl|\,U^\varepsilon_0(u^\varepsilon,w)\,\bigr|\ =\
\bigl|\,\hess_x\phi^\varepsilon(u^\varepsilon,w)\,\bigr|\ \le\ \|\mathcal A^\varepsilon_1(x)\|\,|w| ,
\]
$\|\cdot\|$ the operator norm. In particular $U^\varepsilon_0(u^\varepsilon,w)\to0$ as
$\varepsilon\to0$ for $\mu_0$-a.e.\ $x$, by Lemma~\ref{lem:focalblowup}. The same identity
gives $\bigl|U^\varepsilon_0(u^\varepsilon,u^\varepsilon)+1\bigr|\le\|\mathcal A^\varepsilon_1(x)\|$,
so the tangential entry of $\hess_x(-\phi^\varepsilon)$ tends to $1$.
\end{lem}

\begin{proof}
Work at a point $x$ where $\phi^\varepsilon$ is twice differentiable, $T^\varepsilon$ is
differentiable with differential $\mathcal A^\varepsilon_1=DT^\varepsilon$, the geodesic
$[x\,T^\varepsilon(x)]$ is unique and $x,T^\varepsilon(x)$ are regular --- a set of full
$\mu_0$-measure (Lemma~\ref{lem:focalblowup}(iii)). Put
$\rho(x):=d(x,T^\varepsilon(x))=|\nabla\phi^\varepsilon(x)|$, the McCann displacement
length, so $\nabla\phi^\varepsilon=\rho\,u^\varepsilon$. We compute $\nabla\rho$ twice.

\noindent \emph{First,} from $\tfrac12\rho^2=\tfrac12|\nabla\phi^\varepsilon|^2$ and the chain rule
$\nabla\bigl(\tfrac12|\nabla\phi^\varepsilon|^2\bigr)=\hess\phi^\varepsilon\,\nabla\phi^\varepsilon$,
\[
\rho\,\nabla\rho=U^\varepsilon_0\,\nabla\phi^\varepsilon=U^\varepsilon_0(\rho\,u^\varepsilon),
\qquad\text{i.e.}\qquad \nabla\rho=U^\varepsilon_0\,u^\varepsilon .
\]
\emph{Second,} differentiating $\rho(x)=d(x,T^\varepsilon(x))$ through both slots ($d$
jointly differentiable at $(x,T^\varepsilon(x))$ since the geodesic is unique, $T^\varepsilon$
differentiable at $x$),
\[
\nabla\rho=(\nabla_1 d)(x,T^\varepsilon(x))+(\mathcal A^\varepsilon_1)^{*}(\nabla_2 d)(x,T^\varepsilon(x))
=-\,u^\varepsilon+(\mathcal A^\varepsilon_1)^{*}\nu ,
\]
where $(\nabla_1 d)(x,y)=-\uparrow_x^y=-u^\varepsilon$ and $\nu:=(\nabla_2 d)(x,T^\varepsilon(x))$
is a \emph{unit} vector (first variation at the regular endpoints). Equating the two
expressions,
\[
\bigl(U^\varepsilon_0+\Id\bigr)u^\varepsilon=(\mathcal A^\varepsilon_1)^{*}\nu .
\]
Pairing with any $w$ and using $|\nu|=1$ gives
$\langle(U^\varepsilon_0+\Id)u^\varepsilon,w\rangle=\langle\nu,\mathcal A^\varepsilon_1 w\rangle$,
hence $|\langle(U^\varepsilon_0+\Id)u^\varepsilon,w\rangle|\le\|\mathcal A^\varepsilon_1\|\,|w|$.
Taking $w\perp u^\varepsilon$ kills the term $\langle u^\varepsilon,w\rangle$ and leaves
$|U^\varepsilon_0(u^\varepsilon,w)|\le\|\mathcal A^\varepsilon_1\|\,|w|$; taking
$w=u^\varepsilon$ gives $|U^\varepsilon_0(u^\varepsilon,u^\varepsilon)+1|\le\|\mathcal A^\varepsilon_1\|$.
The convergence is Lemma~\ref{lem:focalblowup}, $\mathcal A^\varepsilon_1\to0$ $\mu_0$-a.e.
\end{proof}

The passage from the Riccati bounds, which live along the rays of the focusing
transports, to the squared-distance bound \eqref{eq:condC}, which concerns every
geodesic, requires two further lemmas: the exact radial structure of the Hessian of
the squared distance, which shows that the tensorial content of \eqref{eq:condC} is
carried entirely by the normal block, and an upgrade lemma converting almost-everywhere
Hessian bounds into support-sense bounds along \emph{every} geodesic.

\begin{lem}[radial structure of the squared distance]\label{lem:radialhess}
Let $p\in X$, $f:=\tfrac12d^2(\cdot,p)$, and let $x\in\operatorname{Reg}f$, $x\neq p$.
Then the geodesic $[xp]$ is unique; write $d:=d(x,p)$ and $u:=\uparrow_x^p\in T_xX$ for
its direction. One has $\nabla f(x)=-d\,u$, and:
\begin{enumerate}
\item[(a)] $\hess_xf(u,u)=1$ \emph{exactly};
\item[(b)] $\hess_xf(u,w)=0$ for every $w\perp u$;
\item[(c)] $\hess_xf(w,w)\le\lambda_\kappa(d)\,|w|^2$ for every $w\perp u$,
\end{enumerate}
where 
$\lambda_\kappa(d):=\sqrt{|\kappa|}\,d\,\coth\bigl(\sqrt{|\kappa|}\,d\bigr)$ for
$\kappa<0$ (and $\lambda_0(d):=1$).
In particular, in the orthogonal splitting $T_xX=\R u\oplus u^\perp$ the Hessian of
$f$ at $x$ is block-diagonal with tangential entry exactly $1$, so that
$\hess_xf\preceq\Id$ holds if and only if the normal block satisfies
$\hess_xf|_{u^\perp}\preceq\Id_{u^\perp}$.
\end{lem}

\begin{proof}
Differentiability of $f$ at $x$ forces uniqueness of the geodesic $[xp]$, and the first variation formula gives $d_xf(w)=-d\,\<u,w\>$ for every geodesic
direction $w$, hence for all $w\in T_xX$ by density of geodesic directions and
linearity of the differential at $x\in\operatorname{Reg}f$; thus $\nabla f(x)=-d\,u$.

(a) Along the geodesic to $p$ one has, \emph{exactly},
$f(\exp_x(su))=\tfrac12(d-s)^2=f(x)-s\,d+\tfrac{s^2}2$ for $0\le s\le d$, subsegments
of $[xp]$ being minimizing. Matching against the second-order expansion
\eqref{eq:hessdefprelim} at $x$ along $w=su$, $s\downarrow0$ gives $\hess_xf(u,u)=1$.

(b)--(c) Let $w_\alpha$ be a geodesic direction at angle $\alpha\in[0,\pi]$ to $u$, and
let $m_\kappa(s)$ be the side of the $\kappa$-model hinge with sides $d$, $s$ and
angle $\alpha$ between them. The $\mathrm{CBB}(\kappa)$ hinge comparison gives
\[
d\bigl(\exp_x(sw_\alpha),\,p\bigr)\ \le\ m_\kappa(s)\qquad(s\ \text{small}),
\]
with equality of values and first derivatives at $s=0$: $m_\kappa(0)=d$,
$m_\kappa'(0)=-\cos\alpha$, matching the first variation. For $\kappa=-1$ (the
general case follows by scaling) the model law of cosines
$\cosh m=\cosh d\cosh s-\sinh d\sinh s\cos\alpha$ yields, differentiating twice at
$s=0$, $m''(0)=\coth(d)\sin^2\alpha$, whence
\begin{align*}
\Bigl(\tfrac{m^2}2\Bigr)''(0)&=m'(0)^2+m(0)\,m''(0)\\
&=\cos^2\alpha+d\coth(d)\sin^2\alpha\\
&=\cos^2\alpha+\lambda_{-1}(d)\sin^2\alpha .
\end{align*}
Comparing the expansion \eqref{eq:hessdefprelim} of $f$ along $w_\alpha$ with that of
$\tfrac12m_\kappa^2$ gives, for every geodesic direction, and then by density of geodesic directions and continuity of both sides in the direction for \emph{every} unit vector $w_\alpha$ at angle $\alpha$ to $u$,
\begin{equation}\label{eq:hingehess}
\hess_xf(w_\alpha,w_\alpha)\ \le\ \cos^2\alpha+\lambda_\kappa(d)\sin^2\alpha .
\end{equation}
Taking $\alpha=\tfrac\pi2$ gives (c). For (b), fix a unit $v\perp u$ and apply
\eqref{eq:hingehess} to $w_\alpha^\pm:=\cos\alpha\,u\pm\sin\alpha\,v$: expanding the
left-hand side bilinearly and using (a),
\[
\pm\,2\cos\alpha\sin\alpha\,\hess_xf(u,v)\ \le\
\sin^2\alpha\,\bigl(\lambda_\kappa(d)-\hess_xf(v,v)\bigr),
\]
and dividing by $\sin\alpha>0$ and letting $\alpha\downarrow0$ (the right-hand side
tends to $0$, $\hess_xf(v,v)$ being finite at $x\in\operatorname{Reg}f$) gives
$\pm\hess_xf(u,v)\le0$, i.e.\ $\hess_xf(u,v)=0$.
\end{proof}

\begin{lem}[limit upgrade: a.e.\ Hessian bounds to support-sense
bounds]\label{lem:limitupgrade}
Let $\Omega\subset X$ be open and relatively compact. For $j\in\mathbb N$ let
$h_j:\Omega\to\R$ be $\lambda_0$-concave (one common $\lambda_0$) with $h_j\to h$
locally uniformly, and let $c_j:\Omega\to\R$ be measurable with $|c_j|\le M$,
$c_j\to c$ $\mu_H$-a.e.\ on $\Omega$, $c$ continuous. Assume that for every $j$
\[
\hess_zh_j\ \preceq\ c_j(z)\,\Id\qquad\text{at $\mu_H$-a.e.\ }z\in\Omega,
\]
in the sense of the a.e.\ expansion \eqref{eq:hessdefprelim}. Then for \emph{every}
unit-speed geodesic $\sigma:[0,L]\to\Omega$,
\[
(h\circ\sigma)''\ \le\ c(\sigma(s))\,ds\qquad\text{as measures on }(0,L).
\]
\end{lem}

Note that the constant sequence $h_j\equiv h$, $c_j\equiv c$ is allowed: an a.e.\ Hessian bound for a single semiconcave function upgrades to the support-sense bound along every geodesic. 

\begin{proof}
\emph{1D facts.} For a $\lambda_0$-semiconcave function $g$ on an interval, the
distributional second derivative is a measure $g''=g''_{ac}\,ds+g''_{s}$ with singular
part $g''_{s}\le0$, and at a.e.\ $s$ the density $g''_{ac}(s)$ coincides with the
second coefficient of the Taylor expansion of $g$ at $s$ (the one-dimensional Alexandrov theorem).

\emph{Good geodesics.} Let $N\subset\Omega$ be the $\mu_H$-null set off which the
hypotheses hold simultaneously for all $j$:
$$
N:=\bigcup_j\{z:\hess_zh_j\ \text{undefined or}\ \not\preceq c_j(z)\Id\}\cup
\{z:c_j(z)\not\to c(z)\},
$$
a countable union of null sets. Call a unit-speed geodesic
$\varsigma$ \emph{good} if $\varsigma(s)\notin N$ for a.e.\ $s$. If $\varsigma$ is
good then, for each $j$ and a.e.\ $s$, the expansion \eqref{eq:hessdefprelim} of $h_j$
at $\varsigma(s)$, restricted along $\varsigma(s+t)=\exp_{\varsigma(s)}(t\dot\varsigma)$,
shows that the second coefficient of the Taylor expansion of $h_j\circ\varsigma$ at $s$ equals
$\hess_{\varsigma(s)}h_j(\dot\varsigma,\dot\varsigma)\le c_j(\varsigma(s))$; by the 1D
facts,
\begin{equation}\label{eq:goodgeodineq}
(h_j\circ\varsigma)''\ \le\ c_j(\varsigma(s))\,ds
\qquad\text{as measures, for every good }\varsigma\ \text{and every }j .
\end{equation}

\emph{Approximation of an arbitrary geodesic by good ones.} Fix $\sigma$ and
$0<s_1<s_2<L$. For small $r>0$ consider the optimal transport between the normalised
restrictions of $\mu_H$ to $B(\sigma(s_1),r)$ and $B(\sigma(s_2),r)$. By
Lemma~\ref{lem:intermediateac} its interpolant measures are absolutely continuous on
compact interior time ranges, so by Lemma~\ref{lem:nullrays} (applied to the
time-independent null set $N$) almost every ray $\varsigma^{(r)}$ of this transport is
good; choose one for each $r$. Its endpoints lie in the two balls, so they converge to
$\sigma(s_1),\sigma(s_2)$ as $r\to0$; the rays have uniformly bounded speed, hence are
precompact (Arzel\`a--Ascoli), and every limit is a geodesic from $\sigma(s_1)$ to
$\sigma(s_2)$. Since $\sigma(s_1),\sigma(s_2)$ are \emph{interior} points of $\sigma$,
that geodesic is unique: a second one would produce, by concatenation with
$\sigma|_{[0,s_1]}$ and $\sigma|_{[s_2,L]}$, a branching pair of geodesics, impossible
in $\mathrm{CBB}$. So $\varsigma^{(r)}\to\sigma|_{[s_1,s_2]}$ uniformly, after the
affine reparametrisations to unit speed, whose factors tend to $1$. For $r$ small the
$\varsigma^{(r)}$ stay in $\Omega$, by the same compactness argument.

\emph{Passage to the limit.} Fix $r$ and let $j\to\infty$ in
\eqref{eq:goodgeodineq} along $\varsigma^{(r)}$: testing against
$0\le\chi\in C^\infty_c$, the left side converges by uniform convergence of
$h_j\circ\varsigma^{(r)}$, the right side by dominated convergence, hence $(h\circ\varsigma^{(r)})''\le c(\varsigma^{(r)}(s))\,ds$.
Now let $r\to0$: $h\circ\varsigma^{(r)}\to h\circ\sigma$ and
$c\circ\varsigma^{(r)}\to c\circ\sigma$ uniformly on $[s_1,s_2]$, so the distributional inequality
passes to the limit: $(h\circ\sigma)''\le c(\sigma(s))\,ds$ on $(s_1,s_2)$.
Exhausting $(0,L)$ finishes the proof.
\end{proof}

We can now establish the squared-distance comparison.

\begin{prop}\label{prop:MDCimpliesC}
Let $(X,d)$ be a finite-dimensional $\mathrm{CBB}(\kappa)$ Alexandrov space,
$\kappa>-\infty$. Assume that for all $\mu_0,\mu_1\in\PR$ there exists, along
$\mu_H$-a.e.\ transport ray, an admissible parallel trivialisation in which block-diagonal matrix displacement convexity holds. Then \eqref{eq:condC} holds: for every $p\in X$ and every geodesic, $\hess\big(\tfrac12 d^2(\cdot,p)\big)\preceq\Id$ in the support sense.
\end{prop}

\begin{proof}
\textbf{Step 1 (Riccati floors at the source).}
Fix a regular point $p$ and the measures $\mu_0,\mu_1^\varepsilon$ of
Lemma~\ref{lem:focalblowup}, with $\mu_0$ the normalised restriction of $\mu_H$ to an
arbitrary ball $B\subset X$. Both are absolutely continuous, so the hypothesis
provides, along $\mu_0$-a.e.\ ray, an admissible parallel trivialisation in which the displacement convexity inequality holds; 
fix it, and work in the corresponding parallel orthonormal frame. 

First the regularity: for fixed
$\xi$, $\tfrac{d}{dt}\<U^\varepsilon_t\xi,\xi\>\le-|U^\varepsilon_t\xi|^2\le0$
distributionally, so $t\mapsto\<U^\varepsilon_t\xi,\xi\>$ is nonincreasing on $(0,1)$;
being also locally bounded there \eqref{eq:twosidedhess}, $U^\varepsilon$ is $BV$ with one-sided limits everywhere, and
its monotone limits at $t=0,1$ are bounded by the endpoint values, which are finite
$\mu_0$-a.e.\ by Lemma~\ref{lem:focalblowup}; hence $U^\varepsilon$ is bounded $BV$ on
$[0,1]$. Note that the one-sided limits of $U^\varepsilon$ at $t=0,1$ are
identified with the endpoint Hessians $U^\varepsilon_0=\hess_x\phi^\varepsilon$ and
$U^\varepsilon_1$ of Lemma~\ref{lem:focalblowup}. This identification is part of the
normalisation fixed at the beginning of this subsection ($U_t=\dot{\mathcal
A}_t\mathcal A_t^{-1}$ with $\mathcal A$ the transport Jacobi field, $\mathcal
A_0=\Id$, $\mathcal A_1=DT^\varepsilon$, whose one-sided derivatives at the endpoints
are the endpoint Hessians). Since $U^\varepsilon_1\prec0$, Lemma~\ref{lem:backwardriccati} with $(t_0,t_1)=(0,1)$ gives, for
$\mu_0$-a.e.\ base point,
\[
U^\varepsilon_0\ \succeq\ \big(U^{\varepsilon\,-1}_1-\Id\big)^{-1}=:W^\varepsilon_0 ,
\]
and by Lemma~\ref{lem:focalblowup}, $W^\varepsilon_0\to-\Id$ as $\varepsilon\to0$.
Moreover, since $U^\varepsilon_1\prec0$ for every $\varepsilon>0$, we have
$(U^\varepsilon_1)^{-1}\prec0$, hence
$(U^\varepsilon_1)^{-1}-\Id\prec-\Id\prec0$, and therefore
\begin{equation}\label{eq:uniformfloor}
W^\varepsilon_0=\big((U^\varepsilon_1)^{-1}-\Id\big)^{-1}\ \in\ (-\Id,\,0),
\qquad\text{in particular }\ W^\varepsilon_0\ \succ\ -\Id ,
\end{equation}
for every $\varepsilon>0$ and $\mu_0$-a.e.\ $x$. Combining \eqref{eq:uniformfloor} with $U^\varepsilon_0
\succeq W^\varepsilon_0$ gives the uniform constant bound
\begin{equation}\label{eq:uniformupper}
\hess_x(-\phi^\varepsilon)=-U^\varepsilon_0\ \preceq\ -W^\varepsilon_0\ \prec\ \Id
\qquad(\mu_H\text{-a.e. }x\in B),
\end{equation}
where $U^\varepsilon_0=\hess_x\phi^\varepsilon$ and $\phi^\varepsilon$ is the McCann velocity
potential of the transport $\mu_0\to\mu_1^\varepsilon$. Note that  \eqref{eq:uniformfloor}--\eqref{eq:uniformupper} are seen blockwise: they hold for the tangential scalar and
the normal block of $-U^\varepsilon_0$, the blocks being taken with respect to the
ray direction $u^\varepsilon(x)$ at $x$. 

\medskip
\textbf{Step 2 (limit and upgrade).}
First, note that the function $-\phi^\varepsilon$ is semiconcave with a modulus uniform in
$\varepsilon$ on a fixed compact neighbourhood of $\overline B$ (the costs
$\tfrac12 d^2(\cdot,\exp_p(\varepsilon w))$ are $\lambda_0$-concave with a single
$\lambda_0=\lambda_0(\kappa,\diam)$ by the $\mathrm{CBB}(\kappa)$ comparison, and a
velocity potential inherits this modulus). By stability of optimal plans
\cite[Thm.~5.20]{Villani09} together with $\mathrm{CBB}$-uniqueness
\cite[Thm.~1.1]{bertrand2008existence}, $-\phi^\varepsilon\to\tfrac12 d^2(\cdot,p)$
locally uniformly as $\varepsilon\to0$.

Now, Step~1 gives, at $\mu_H$-a.e.\ $x\in B$, the two blockwise
bounds relative to the ray direction $u^\varepsilon(x)$,
\[
\tilde c:=\<\hess_x(-\phi^\varepsilon)u^\varepsilon,u^\varepsilon\>\le1,
\qquad
\tilde V:=\hess_x(-\phi^\varepsilon)\big|_{(u^\varepsilon)^\perp}\prec\Id ,
\]
with the mixed entries $\tilde m:=\hess_x(-\phi^\varepsilon)(u^\varepsilon,\cdot)|_{(u^\varepsilon)^\perp}$
a priori uncontrolled by the block hypothesis. They are, however, controlled by the
focal degeneration: Lemma~\ref{lem:focalmixed} gives
$|\tilde m|\le\eta^\varepsilon(x):=\|\mathcal A^\varepsilon_1(x)\|$ at $\mu_0$-a.e.\ $x$,
with $\eta^\varepsilon\to0$ $\mu_0$-a.e.\ as $\varepsilon\to0$. We upgrade the two
blockwise bounds plus this mixed control to an \emph{isotropic} a.e.\ bound. In the
splitting-adapted frame, $-U^\varepsilon_0=\big(\begin{smallmatrix}\tilde c&\tilde m^{\top}\\
\tilde m&\tilde V\end{smallmatrix}\big)$. If $\eta^\varepsilon(x)=0$ this is block-diagonal
and $\preceq\Id$; if $\eta:=\eta^\varepsilon(x)>0$, set $s:=\eta$ and consider
$M:=(1+s)\Id-(-U^\varepsilon_0)$, whose blocks are $M_{11}=1+s-\tilde c\ge s$,
$M_{22}=(1+s)\Id-\tilde V\succeq s\Id\succ0$ (as $\tilde V\prec\Id$), $M_{12}=-\tilde m$.
The Schur complement of $M_{22}$ is
\[
M_{11}-\tilde m^{\top}M_{22}^{-1}\tilde m\ \ge\ s-\frac{|\tilde m|^{2}}{s}\ \ge\
s-\frac{\eta^{2}}{s}\ =\ 0 ,
\]
so $M\succeq0$, i.e.
\begin{equation}\label{eq:isotropicsource}
\hess_x(-\phi^\varepsilon)\ \preceq\ \bigl(1+\eta^\varepsilon(x)\bigr)\,\Id
\qquad(\mu_H\text{-a.e. }x\in B).
\end{equation}
This is the right input for Lemma~\ref{lem:limitupgrade}. Set
$c_\varepsilon(x):=\min\{1+\eta^\varepsilon(x),\,\lambda_0\}$, where $\lambda_0$ is the
uniform semiconcavity modulus of the cost (so $\hess_x(-\phi^\varepsilon)\preceq\lambda_0\Id$
always): then $c_\varepsilon$ is measurable, $1\le c_\varepsilon\le\lambda_0$,
$\hess_x(-\phi^\varepsilon)\preceq c_\varepsilon(x)\Id$ a.e.\ (by \eqref{eq:isotropicsource}
where $1+\eta^\varepsilon\le\lambda_0$, and by semiconcavity otherwise), and
$c_\varepsilon\to1$ $\mu_0$-a.e.\ (since $\eta^\varepsilon\to0$), a continuous limit.
Applying Lemma~\ref{lem:limitupgrade} on $\Omega=B$ with $h_j:=-\phi^{\varepsilon_j}$,
$c_j:=c_{\varepsilon_j}$, $c\equiv1$, $M=\lambda_0$ yields
$\big(\tfrac12 d^2(\sigma(s),p)\big)''\le1$ as measures along every geodesic $\sigma$ in
$B$, i.e.\ \eqref{eq:condC}. 
\end{proof}

\subsection{Squared-distance bound implies nonnegative curvature}\label{ss:CtoAlex}

It remains to recall that the squared-distance comparison \eqref{eq:condC} characterises
nonnegative Alexandrov curvature; combined with
Proposition~\ref{prop:MDCimpliesC} this completes the proof of
Theorem~\ref{thm:displacementconveximpliesAlex} (assembled at the end of the subsection).

\begin{prop}\label{prop:CtoAlex}
Let $(X,d)$ be a finite-dimensional Alexandrov space with curvature bounded below by some
$\kappa>-\infty$, satisfying the squared-distance comparison \eqref{eq:condC}: for every
$p\in X$ and every geodesic, $\hess\big(\tfrac12 d^2(\cdot,p)\big)\preceq\Id$ in the support
sense. Then $(X,d)$ has nonnegative Alexandrov curvature.
\end{prop}

\begin{proof}
By hypothesis $(X,d)$ is a finite-dimensional Alexandrov space with $\curv\ge\kappa$; in
particular it is locally compact, locally geodesic and intrinsic, angles between geodesics
exist, the first variation formula holds, and for $p\ne q$ and a geodesic $[qr]$ the angle
$\angle pqr=\lim_{t\to q}\widetilde\angle_0\,pqt$ exists.

Fix $x\in X$ and a neighbourhood $U_x$ in which the $\curv\ge\kappa$ comparison holds. Fix
$p\in U_x$ and a minimal geodesic $[qr]\subset U_x$ with $p\ne q$. Parametrize $[qr]$ by
arclength $s\mapsto\sigma(s)$, $\sigma(0)=q$, with unit initial direction
$v=\,\uparrow_q^{r}\in T_qX$, $|v|=1$, and write $a:=|pq|>0$.

The hypothesis \eqref{eq:condC} at $q$, applied to the geodesic $\sigma$ and the point $p$,
reads $\hess_q\big(\tfrac12 d^2(\cdot,p)\big)(v,v)\le|v|^2=1$ in the support sense; that is,
there is $A\in\R$ with
\[
d^2(\sigma(s),p)\le a^2+As+s^2+o(s^2).
\]

By the first variation formula in Alexandrov spaces (see \textit{e.g.} \cite{AlexanderKapovitchPetrunin}),
\[
\displaystyle\frac{d}{ds}\Big|_{0^+}d(\sigma(s),p)=-\cos\angle pqr=:-\cos\alpha,
\]
so
$$A=\frac{d}{ds}\big|_{0^+}d^2(\sigma(s),p)=-2a\cos\alpha$$ with
$\alpha=\angle pqr=\lim_{t\to q}\widetilde\angle_0\,pqt$. Therefore
\[
d^2(\sigma(s),p)\ \le\ a^2-2a\cos\alpha\,s+s^2+o(s^2).
\]
Let $\triangle\bar p\bar q\bar s\subset\mathbb S^2_0=\R^2$ be the comparison triangle with
$|\bar p\bar q|=a$, $|\bar q\bar s|=s$ and $\angle\bar p\bar q\bar s=\alpha$. By the
Euclidean law of cosines $|\bar p\bar s|^2=a^2+s^2-2as\cos\alpha$, whence
\[
d^2(\sigma(s),p)\le|\bar p\bar s|^2+o(s^2),
\qquad\text{i.e.}\qquad
|p\,\sigma(s)|\ \le\ |\bar p\bar s|+o\bigl(|q\,\sigma(s)|^2\bigr),
\]
using $|\bar p\bar s|\to a>0$. As $\alpha=\angle pqr$, this is exactly condition (C)
(equation~(1.8)) of \cite{HuSuWang} with $k=0$, for the hinge $(p,[qr])$ at $q$.

Since $p$ and $[qr]$ were arbitrary in $U_x$, condition (C) holds throughout $U_x$ with
$k=0$. Moreover, since $X$ has $\curv\ge\kappa$, angles are well defined and for $q$
interior to a geodesic $[rr']$ one has $\angle pqr+\angle pqr'=\pi$; thus condition (4-1) of
\cite{HuSuWang} holds in $U_x$. By Remark~1.2 of \cite{HuSuWang} (equivalently Theorem~B
through condition (B2)), a locally geodesic intrinsic metric space satisfying (4-1) and (C)
with $k=0$ has curvature $\ge0$ in the Alexandrov sense on $U_x$. As $x\in X$ was arbitrary,
$(X,d)$ has nonnegative Alexandrov curvature.
\end{proof}

\begin{proof}[Proof of Theorem~\ref{thm:displacementconveximpliesAlex}]
By Proposition~\ref{prop:MDCimpliesC}, the matrix displacement
convexity of the block-diagonal entropy tensor, for all $\mu_0,\mu_1\in\PR$, in some admissible
parallel trivialisation along $\mu_H$-a.e.\ transport ray, implies the squared-distance comparison \eqref{eq:condC}.
By Proposition~\ref{prop:CtoAlex}, \eqref{eq:condC} implies that $(X,d)$ has nonnegative Alexandrov curvature, i.e.\ $\kappa$ may be taken to be $0$.
\end{proof}

\appendix

\section{Transport rays and null sets}\label{app:rays}

This appendix collects the two measure-theoretic facts about transport rays used several times in this paper: absolute continuity of the displacement interpolation
at interior times, and the resulting Fubini statement for time-dependent null sets. Both
are essentially contained in \cite[\S\S1,3]{petrunin2010alexandrov}; but for completeness, we spell the proofs out.

\begin{lem}[absolute continuity at intermediate times]\label{lem:intermediateac}
Let $(X,d)$ be a finite-dimensional $\mathrm{CBB}(\kappa)$ space, $\kappa>-\infty$, let
$\mu_0,\mu_1$ be compactly supported probability measures with $\mu_0,\mu_1\ll\mu_H$,
and let $(\mu_t)_{t\in[0,1]}$ be the displacement interpolation between them. Then
$\mu_t\ll\mu_H$ for every $t\in[0,1]$.
\end{lem}

\begin{proof}
Fix $t\in(0,1)$. Let $\Pi$ be the dynamical optimal
coupling, supported on the space $\Gamma$ of minimizing geodesics
$\gamma:[0,1]\to X$ with $\mu_s=(e_s)_\#\Pi$, $e_s(\gamma)=\gamma(s)$, and let
$\psi_s:=H_s\psi$ be the Hamilton--Jacobi shifts of the optimal price function, as in
\cite[\S3]{petrunin2010alexandrov}. Since $\tfrac12d^2(\cdot,y)$ is
$\lambda_0$-concave on the (compact) transport set, with
$\lambda_0=\lambda_0(\kappa,\mathrm{diam})$, the Hopf--Lax formula exhibits $\psi_s$ as
an infimum of $\tfrac{\lambda_0}s$-concave functions, so $(\psi_s)_{s\in(0,1]}$ is a
$\tfrac{\lambda_0}s$-concave family in the sense of
\cite[\S1]{petrunin2010alexandrov}. By the first-variation identity
$\<\gamma^+(s),v\>=d_{\gamma(s)}\psi_s(v)$ of \cite[\S3]{petrunin2010alexandrov}, every $\gamma\in\Gamma$ is a $\psi_s$-gradient curve
on $(0,1]$.
Consequently the gradient-flow maps $\Psi_{t,1}$ of the family
$(\psi_s)$ satisfy $\gamma(1)=\Psi_{t,1}(\gamma(t))$ for all $\gamma\in\Gamma$, i.e.
\[
\mu_1=(\Psi_{t,1})_\#\,\mu_t ,
\]
and by \cite[Prop.-def.~1.2]{petrunin2010alexandrov} the map $\Psi_{t,1}$ is
$L$-Lipschitz on the transport set with
$L=\exp\bigl(\int_t^1\tfrac{\lambda_0}s\,ds\bigr)=t^{-\lambda_0}$.

Now let $N$ be a Borel set with $\mu_H(N)=0$; we must show $\mu_t(N)=0$. Restricting
to the compact transport set, the image $\Psi_{t,1}(N)$ is contained in a Borel set of
$\mu_H$-measure $0$ (since $\Psi_{t,1}$ is Lipschitz). Then
\[
\mu_t(N)\ \le\ \mu_t\bigl(\Psi_{t,1}^{-1}(\Psi_{t,1}(N))\bigr)
=\mu_1\bigl(\Psi_{t,1}(N)\bigr)=0
\]
since $\mu_1\ll\mu_H$. Hence $\mu_t\ll\mu_H$. 
\end{proof}

\begin{lem}[time-dependent null sets are a.e.\ avoided]\label{lem:nullrays}
In the setting of Lemma~\ref{lem:intermediateac}, parametrise the transport rays
$\gamma_x$ over the starting points $x$ ($\mu_0$-a.e.\ defined) by a parameter interval
$(0,a)$, affinely reparametrising McCann time. Let $(N_t)_{t\in(0,a)}$ be a family of
Borel sets with $\mu_H(N_t)=0$ for a.e.\ $t$, such that
$\{(t,y):y\in N_t\}$ is product-measurable. Then for $\mu_0$-a.e.\ $x$,
\[
\mathrm{Leb}\bigl(\{t\in(0,a):\gamma_x(t)\in N_t\}\bigr)=0 .
\]
In particular, for a single $\mu_H$-null Borel set $N$ (constant family), $\mu_0$-a.e.\
ray meets $N$ in a set of times of measure zero; equivalently
$\H^1(\gamma\cap N)=0$, the rays having constant speed.
\end{lem}

\begin{proof}
The function $(t,x)\mapsto\ind_{N_t}(\gamma_x(t))$ is jointly measurable: $(t,x)\mapsto
\gamma_x(t)$ is continuous in $t$ for fixed $x$ and measurable in $x$ for fixed $t$
(measurability of the ray map; the optimal map and its interpolants are
$\mu_0$-measurable), hence jointly measurable (Carath\'eodory), and the indicator of
the product-measurable set $\{(t,y):y\in N_t\}$ composes measurably. Tonelli then
gives
\begin{align*}
\int\mathrm{Leb}\bigl(\{t:\gamma_x(t)\in N_t\}\bigr)\,d\mu_0(x)
&=\int_0^a\mu_0\bigl(\{x:\gamma_x(t)\in N_t\}\bigr)\,dt\\
&=\int_0^a\mu_{s(t)}\bigl(N_t\bigr)\,dt\\
&=0 ,
\end{align*}
where $s(t)\in[0,1]$ is the affine reparametrisation and we used Lemma~\ref{lem:intermediateac}: $(x\mapsto\gamma_x(t))_\#\mu_0=\mu_{s(t)}\ll\mu_H$ together with $\mu_H(N_t)=0$ for a.e.\ $t$.
\end{proof}

\bibliographystyle{amsplain}
\bibliography{ref}

\providecommand{\bysame}{\leavevmode\hbox to3em{\hrulefill}\thinspace}
\providecommand{\MR}{\relax\ifhmode\unskip\space\fi MR }
\providecommand{\MRhref}[2]{%
  \href{http://www.ams.org/mathscinet-getitem?mr=#1}{#2}
}
\providecommand{\href}[2]{#2}
\begin{thebibliography}{10}

\bibitem{aishwarya2025sectional}
Gautam Aishwarya, Liran Rotem, and Yair Shenfeld, \emph{Sectional curvature and
  matrix displacement convexity}, arXiv preprint arXiv:2509.23399 (2025).

\bibitem{AlexanderKapovitchPetrunin}
Stephanie Alexander, Vitali Kapovitch, and Anton Petrunin, \emph{{Alexandrov
  Geometry: Foundations}}, Graduate Studies in Mathematics, vol. 236, American
  Mathematical Society, 2024.

\bibitem{ambrosio2018dc}
Luigi Ambrosio and J{\'e}r{\^o}me Bertrand, \emph{Dc calculus}, Mathematische
  Zeitschrift \textbf{288} (2018), no.~3, 1037--1080.

\bibitem{bertrand2008existence}
J{\'e}r{\^o}me Bertrand, \emph{{Existence and uniqueness of optimal maps on
  Alexandrov spaces}}, Advances in Mathematics \textbf{219} (2008), no.~3,
  838--851.

\bibitem{BuragoBuragoIvanov01}
Dmitri Burago, Yuri Burago, and Sergei Ivanov, \emph{{A Course in Metric
  Geometry}}, Graduate Studies in Mathematics, vol.~33, American Mathematical
  Society, 2001.

\bibitem{BuragoGromovPerelman92}
Yu~Burago, Mikhail Gromov, and Gregory Perel'man, \emph{{A}. {D}. {A}lexandrov
  spaces with curvature bounded below}, Russian mathematical surveys
  \textbf{47} (1992), no.~2, 1.

\bibitem{Cordero-ErasquinMcCannSchmuckenschlager01}
Dario Cordero-Erausquin, Robert~J. McCann, and Michael Schmuckenschl\"ager,
  \emph{A {R}iemannian interpolation inequality \`a{} la {B}orell, {B}rascamp
  and {L}ieb}, Invent. Math. \textbf{146} (2001), no.~2, 219--257. \MR{1865396}

\bibitem{HuSuWang}
Shengqi Hu, Xiaole Su, and Yusheng Wang, \emph{{New definitions of Alexandrov
  space and applications}}, arXiv preprint arXiv:2102.02112 (2023).

\bibitem{LottVillani09}
John Lott and C{\'e}dric Villani, \emph{Ricci curvature for metric-measure
  spaces via optimal transport}, Annals of Mathematics (2009), 903--991.

\bibitem{otsu30differential}
Yukio Otsu, \emph{{Differential Geometric Aspects of Alexandrov Spaces}},
  Mathematical Sciences Research Institute Publications, p.~135–148,
  Cambridge University Press, 1997.

\bibitem{OtsuShioya94}
Yukio Otsu and Takashi Shioya, \emph{{The Riemannian structure of Alexandrov
  spaces}}, Journal of Differential Geometry \textbf{39} (1994), no.~3,
  629--658.

\bibitem{Perelman_DC}
Grigori Perelman, \emph{{DC structure on Alexandrov space}}, Preprint, 1995.

\bibitem{PerelmanPetruninQG}
Grigori Perelman and Anton Petrunin, \emph{{Quasigeodesics and gradient curves
  in Alexandrov spaces}}, Preprint (1994).

\bibitem{petrunin1998parallel}
Anton Petrunin, \emph{{Parallel transportation for Alexandrov space with
  curvature bounded below}}, Geometric \& Functional Analysis GAFA \textbf{8}
  (1998), no.~1, 123--148.

\bibitem{petrunin2010alexandrov}
\bysame, \emph{{Alexandrov meets Lott--Villani--Sturm}}, arXiv preprint
  arXiv:1003.5948 (2010).

\bibitem{Shenfeld24}
Yair Shenfeld, \emph{Matrix displacement convexity along density flows}, Arch.
  Ration. Mech. Anal. \textbf{248} (2024), no.~5, Paper No. 74, 41.
  \MR{4791533}

\bibitem{Sturm06(i)}
Karl-Theodor Sturm, \emph{On the geometry of metric measure spaces. i}, Acta
  Math \textbf{196} (2006), 65--131.

\bibitem{Villani09}
C\'edric Villani, \emph{Optimal transport}, Grundlehren der mathematischen
  Wissenschaften [Fundamental Principles of Mathematical Sciences], vol. 338,
  Springer-Verlag, Berlin, 2009, Old and new. \MR{2459454}

\bibitem{vonRenesseSturm05}
Max-K. von Renesse and Karl-Theodor Sturm, \emph{Transport inequalities,
  gradient estimates, entropy, and {R}icci curvature}, Comm. Pure Appl. Math.
  \textbf{58} (2005), no.~7, 923--940. \MR{2142879}

\end{thebibliography}
\end{document}